\documentclass{article}
\usepackage[utf8]{inputenc}
\usepackage{amsmath,amssymb,amsthm, bbold, xfrac, graphicx, tabu, tikz, framed, caption, subcaption, verbatim, url, upgreek, wasysym, bbm, dictsym, staves, physics, algorithmic, calligra, mathrsfs, breqn, placeins, enumitem, hyperref, float}
\hypersetup{
    colorlinks=true,
    citecolor=blue,
    urlcolor=violet,
    linkcolor=red
}
\usepackage[margin=1 in]{geometry}
\graphicspath{ {./Images/} }

\numberwithin{equation}{section}
\numberwithin{figure}{section}

\newtheorem{theorem}{Theorem}[section]
\newtheorem*{theorem*}{Theorem}

\newtheorem{lemma}[theorem]{Lemma}
\newtheorem{prop}[theorem]{Proposition}
\newtheorem{corollary}[theorem]{Corollary}

\theoremstyle{definition}
\newtheorem{definition}[theorem]{Definition}
\newtheorem{example}[theorem]{Example}

\newtheorem{remark}[theorem]{Remark}

\newcommand\inner[1]{\langle #1\rangle}

\newcommand{\indicator}{\mathbb{1}}
\newcommand{\indicatorWithSetBrackets}[1]{\indicator_{\left\{ #1 \right\}}}

\newcommand{\EHI}{\hyperref[HarnackDef]{\mathrm{EHI}}}
\newcommand{\LargeEHI}{\hyperref[LargeHarDef]{\mathrm{LargeEHI}}}
\newcommand{\SmallEHI}{\hyperref[SmallHarDef]{\mathrm{SmallEHI}}}

\newcommand{\pFlip}{\mathbf{pFlip}}
\newcommand{\Uhat}{\widehat{\mathcal{U}}}
\newcommand{\PHI}{\mathrm{PHI}}

\title{\textbf{Counterexamples to elliptic Harnack inequality for isotropic unimodal L\'{e}vy processes}}
\author{Jens Malmquist\footnote{University of British Columbia, Vancouver, CA. Email: \url{jens@math.ubc.ca}} \quad Mathav Murugan\footnote{University of British Columbia, Vancouver, CA. Research partially supported by NSERC and the Canada research chairs program.
}}

\begin{document}

\maketitle

\begin{abstract}
We construct the first examples of subordinated Brownian motion (SBM) that do not satisfy the elliptic Harnack inequality. In our first theorem, we show that if $X=(X_t)_{t \geq 0}$ is an isotropic unimodal L\'{e}vy process, and $X$ satisfies certain criteria (involving the jump kernel of $X$ and the distribution of the location of the process upon first exiting balls of various sizes) then $X$ does not satisfy EHI. (Note that every SBM is an isotropic unimodal L\'{e}vy process.) We then check that many specific SBMs do indeed satisfy our criteria, and thus do not satisfy EHI.
\end{abstract}


\tableofcontents

\section{Introduction}

Subordinated Brownian motion plays a central role in the understanding of jump processes as it serves as a model example of jump processes analogous to the role played by Brownian motion for diffusions.
Heat kernel estimates and Harnack inequalities are often easier to obtain for subordinated Brownian motion. Using recent progress in stability of Harnack inequalities and heat kernel bounds for jump processes this leads to similar estimates for a large family of jump processes \cite[Section 5.2]{CKW}.

Harnack inequalities are the subject of significant research in probability, harmonic analysis, and partial differential equations. The \textit{elliptic Harnack inequality} ($\EHI$, see Section \ref{HarmonicAndHarnackSubsection}) applies to harmonic functions, and says that if a non-negative function $h$ is harmonic on a ball $B$, then for any points $x, y$ on a smaller ball concentric to $B$, the ratio $h(x)/h(y)$ is bounded both above and below. Applications of elliptic Harnack inequality include regularity of harmonic functions, estimates of heat kernel and Green's function.
There is an analogous inequality, the \textit{parabolic Harnack inequality} ($\PHI$), which implies H\"{o}lder continuity for caloric functions. Since every harmonic function lifts to a caloric function, the parabolic Harnack inequality implies the elliptic Harnack inequality.

Recall that a \textit{L\'{e}vy process} is a stochastic process with independent and stationary increments.
A \textit{subordinated Brownian motion} (or \textit{SBM}) is a process of the form $X=(X_t)_{t \geq 0} = (W(S_t))_{t \geq 0}$, where $S$ is a L\'{e}vy process on $[0, \infty)$ such that $S_0=0$, and $W$ is a standard Brownian motion on $\mathbb{R}^d$, independent of $S$.

There have been many results (\cite{KM}, \cite{G}, \cite{CKW}) showing that various classes of SBMs satisfy the elliptic Harnack inequality (or close variations of it). The settings of \cite{G} and (especially) \cite{CKW} are more general, but both apply to many SBMs.  Despite this progress, the following question remained open: \emph{Does every suborinated Brownian motion satisfy the elliptic Harnack inequality?} The goal of this paper is to provide a negative answer to this question. In this paper, we produce many counterexamples and give criteria that can be checked to verify that a subordinated Brownian motion fails to satisfy the $\EHI$. Our work can be viewed as a step towards the following motivating question: \emph{find necessary and sufficient conditions for a subordinated Brownian motion to satisfy the elliptic Harnack inequality.}

In fact, our results do not apply only to SBMs, but to the larger class of \textit{isotropic unimodal L\'{e}vy processes}. A \textit{L\'{e}vy process} is a stochastic process $X=(X_t)_{t \geq 0}$ with independent and stationary increments; in other words, a process such that
\begin{equation*}
    X_t-X_s \overset{d}{=} X_{t-s} \qquad\mbox{for all $t \geq s \geq 0$}
\end{equation*}
and
\begin{equation*}
    \left\{ X_{t_1}-X_{t_0}, X_{t_2}-X_{t_1}, \dots, X_{t_n}-X_{t_{n-1}} \right\} \qquad\mbox{are independent}\qquad\mbox{for all $t_0 \leq t_1 \leq \cdots \leq t_{n-1} \leq t_n$}.
\end{equation*}
Such a process is said to be \textit{isotropic unimodal} if for all $t>0$, there exists a non-increasing function $m_t : (0, \infty) \to [0, \infty)$ such that the increments of $X$ follow distribution
\begin{equation*}
    \mathbb{P}(X_t-X_0 \in A) = \int_A m_t(|x|) \, dx \qquad\mbox{for all measurable $A$ such that $0 \notin A$}.
\end{equation*}
More generally, a measure $\mu$ is called \textit{isotropic unimodal} if there exists a non-increasing $m:(0, \infty) \to [0, \infty)$ such that $\mu(dx) = m(|x|) \, dx$ for all $x \neq 0$, and a process is isotropic unimodal if its increments are.

Bass and Chen \cite[Section 3]{BC} give an example of a L\'{e}vy process on $\mathbb{R}^d$ that does not satisfy the elliptic Harnack inequality. The process they consider has a high degree of non-regularity (it is not isotropic unimodal, and its L\'{e}vy measure is singular with respect to the Lebesgue measure), and they exploit this in their proof. Grzywny and Kwa\'{s}nicki \cite[Example 5.5]{GK} give another process, this one with more regularity, that also fails to satisfy the elliptic Harnack inequality. The process considered in \cite[Example 5.5]{GK} is an isotropic unimodal L\'{e}vy process, but not a subordinated Brownian motion as the jump sizes are uniformly bounded from above which is crucial to their proof.

Our main argument is similar to those of \cite[Section 3]{BC} and \cite[Example 5.5]{GK}, with a significant amount of additional work to adapt them for processes with a higher degree of regularity. Our proof provides some insight into what kinds of non-regularity are necessary for $\EHI$ to be violated. Our most highly-regular counterexample is Example \ref{ContinuousExample}: a subordinated Brownian motion, with a subordinator whose L\'{e}vy measure is absolutely continuous with respect to the Lebesgue measure, with a density that is monotonically decreasing.

Let $X=(X_t)_{t \geq 0}$ be an isotropic unimodal L\'{e}vy process on $\mathbb{R}^d$. We will assume that $X$ is right-continuous, with left-limits (\textit{cadlag}). For all $t>0$, we will use the notation $X_{t-}$ or $X(t-)$ to denote the left-limit $\lim_{s \nearrow t} X_s$.

For all $x \in \mathbb{R}^d$, let $\mathbb{P}_x$ be the probability measure $\mathbb{P}\left( \cdot | X_0=x \right)$. Let $\mathbb{E}_x$ be the corresponding expectation.

For all $x_0 = \left(x^{(0)}_1, \dots, x^{(0)}_d \right) \in \mathbb{R}^d$ and $r>0$, let $B(x_0, r)$ denote the open Euclidean ball
\begin{equation*}
    B(x_0, r) := \left\{ (x_1, \dots, x_d) : \sqrt{\sum_{j=1}^d \left(x_j-x^{(0)}_j \right)^2} < r \right\}.
\end{equation*}
For all $x_0 = \left(x^{(0)}_1, \dots, x^{(0)}_{d-1} \right) \in \mathbb{R}^{d-1}$ and $r>0$, it will also help to consider the $(d-1)$-dimensional ball
\begin{equation} \label{Bd-1}
    B^{(d-1)}(x_0, r) := \left\{ (x_1, \dots, x_{d-1}) : \sqrt{\sum_{j=1}^{d-1} \left(x_j-x^{(0)}_j \right)^2} < r \right\}.
\end{equation}

We denote exit times of $X$ by
\begin{equation} \label{ExitTimeDef}
    \tau_U := \inf\{ t\geq 0: X_t \notin U\} \qquad\mbox{for all $U \subseteq \mathbb{R}^d$}.
\end{equation}

\subsection{Jump kernel} \label{JumpKernelSubsection}

Before we state our main result, let us briefly discuss the jump kernel.

Let $X=(X_t)_{t \geq 0}$ be an isotropic unimodal L\'{e}vy process on $\mathbb{R}^d$. Let $(\mathcal{E}, \mathcal{F})$ be the associated regular Dirichlet form on $(\mathcal{E}, \mathcal{F})$.

By the Beurling-Deny formula \cite[Theorem 3.2.1]{FOT}, $\mathcal{E}$ can be decomposed into a strongly local component, a jumping component, and a killing component:
\begin{equation*}
    \mathcal{E}(f, g) = \mathcal{E}^c(f, g) + \int_{\mathbb{R}^d \times \mathbb{R}^d \setminus \mathrm{diag}} (f(x)-f(y))(g(x)-g(y)) \, \widehat{J}(dx, dy) + \int_{\mathbb{R}^d} f(x) g(x) \, \widehat{k}(dx)
\end{equation*}
where $\mathrm{diag}$ denotes the diagonal of $\mathbb{R}^d \times \mathbb{R}^d$, $\mathcal{E}^c$ is a strongly local symmetric form, $\widehat{J}$ is a symmetric non-negative Radon measure (which we call the \textit{jumping measure}), and $\widehat{k}$ is a non-negative Radon measure (which we call the \textit{killing measure}).

Since $X$ is a L\'{e}vy process, $X$ is never killed, so the killing measure is identically $0$. However, $X$ can still have both a strongly local (diffusion) component and a jump component.

Since $X$ is isotropic unimodal, it follows from the $(1)\Longrightarrow(3)$ implication of \cite[Proposition on page 488]{W} that there exists a non-increasing function $j : (0, \infty) \to [0, \infty)$ such that $\widehat{J}(dx, dy) = j(|x-y|) \, dx \, dy$.

Given two distinct points $x, y \in \mathbb{R}^d$, let $J(x, y) := j(|x-y|)$. We refer to the function $J$ on $\mathbb{R}^d \times \mathbb{R}^d \setminus \mathrm{diag}$ as the \textit{jump kernel} of $X$.

Given a point $x \in \mathbb{R}^d$ and a measurable set $\mathcal{U} \subseteq \mathbb{R}^d \setminus \{x\}$, let
\begin{equation*}
    J(x, \mathcal{U}) := \int_U J(x, y) \, dy.
\end{equation*}

If $X(t) \neq X(t-)$ for some $t>0$, we say that the process $X$ has a jump of \textit{displacement} $X(t)-X(t-)$ at time $t$. We may also say that $X$ has a jump of \textit{magnitude} $|X(t)-X(t-)|$ at time $t$.

There is a probabilistic interpretation of the jump kernel. For all measurable $A \subseteq \mathbb{R}^d \setminus 0$ such that $0<J(0, A) < \infty$, jumps with displacement in $A$ occur according to a Poisson process with rate $J(0, A)$.

\subsection{Harmonic functions and elliptic Harnack inequality} \label{HarmonicAndHarnackSubsection}

There are many formulations of harmonicity in the context of general symmetric Hunt processes (see \cite{C}). For simplicity, we will use the same formulation as \cite{KM}.

Let $X$ be a non-trivial L\'{e}vy process on $\mathbb{R}^d$. By non-trivial we mean that the process in not identically zero.
Let $D$ be an open subset of $\mathbb{R}^d$. We say that a function $h : \mathbb{R}^d \to \mathbb{R}$ is harmonic (with respect to $X$) on $D$ if for all open sets $E$ whose closure is compact and contained in $D$, we have the mean-value property
\begin{equation*}
    h(x) = \mathbb{E}_x \left[ h(X_{\tau_E}) \right] \qquad\mbox{for all $x \in E$}.
\end{equation*}

We say that $X$ satisfies the \textit{elliptic Harnack inequality} ($\EHI$) if there exist a $C>0$ and a $\kappa \in (0, 1)$ such that for all $x_0 \in \mathbb{R}^d$ and $r>0$, if $h$ is a non-negative function that is harmonic on $B(x_0, r)$, then
\begin{equation}\label{HarnackDef}
    h(x) \leq C h(y) \qquad\mbox{for all $x, y \in B(x_0, \kappa r)$}.
\end{equation}
Note that $\EHI$ does not depend on $\kappa$. If $\EHI$ holds some $\kappa \in (0, 1)$, it will still hold if $\kappa$ is replaced by any other value in $(0, 1)$ (and $C$ is replaced with an appropriate value).

\subsection{Subordinated Brownian motion}

Recall that we are particularly interested in the case where $X$ is a subordinated Brownian. Therefore, let us review some of the basic properties of subordinated Brownian motions.

Let $S=(S_t)_{t \geq 0}$ be a non-decreasing, right-continuous L\'{e}vy process on $[0, \infty)$, such that $S_0=0$. By a L\'{e}vy process, we mean a stochastic process with independent and stationary increments.

Let us review some of the basic properties of the $S$. There exists a function $\phi : (0, \infty) \to (0, \infty)$ called the \textit{Laplace exponent} such that
\begin{equation*}
    \mathbb{E} e^{-\lambda S_t} = e^{-\phi(\lambda)\, t}  \qquad\mbox{for all $\lambda>0$, $t>0$}.
\end{equation*}
The Laplace exponent is a \textit{Bernstein function}, meaning that $\phi$ is smooth, $\phi^{(k)} \geq 0$ for all $k \in \{0, 1, 3, 5, 7, \dots\}$, and $\phi^{(k)} \leq 0$ for all $k \in \{2, 4, 6, \dots\}$. Also, $\phi$ has a representation of the form
\begin{equation*}
    \phi(\lambda) = \gamma \lambda + \int_{(0, \infty)} (1-e^{-\lambda x}) \, \mu(dx)
\end{equation*}
for some $\gamma \geq 0$ (which we call the \textit{drift}) and some measure $\mu$ on $(0, \infty)$ (which we call the \textit{L\'{e}vy measure}) such that $\int_{(0, \infty)} (1 \wedge x) \, \mu(dx) < \infty$. The drift and L\'{e}vy measure of $S$ have a probabilistic interpretation. The process $S$ increases due to both continuous linear growth (with rate $\gamma$) and jumps. For all measurable $A \subseteq (0, \infty)$ such that $0<\mu(A)<\infty$, the number of jumps of magnitude in $A$ that occur by time $t$ is Poisson($\mu(A) \cdot t$). For all $t>0$, $S_t$ is equal to $\gamma t$ plus the sum of the magnitudes of all the jumps that have occurred by time $t$.

Let $W=(W(t))_{t \geq 0}$ be a standard Brownian motion on $\mathbb{R}^d$, independent from $S$. The heat kernel of $W$ is
\begin{equation*}
    p^W(t, x, y) = (2\pi t)^{-d/2} \exp\left( - \frac{|x-y|^2}{2t} \right).
\end{equation*}

Let $X:= W(S_t)$ for all $t$. We refer to the process $X=(X_t)_{t \geq 0}$ as a \textit{subordinated Brownian motion} (SBM), and we refer to $S$ as the \textit{subordinator} of $X$.
Note that a SBM is uniquely determined by $\gamma$ (the drift of its subordinator) and $\mu$ (the L\'{e}vy measure of its subordinator).

If $\gamma=0$ and $\mu((0, \infty))>0$, then $X$ is a pure-jump process on $\mathbb{R}^d$. If $\gamma>0$ but $\mu((0,\infty))=0$, then $X$ is a Brownian motion, with $\gamma$ times the speed of a standard Brownian motion. If $\gamma>0$ and $\mu((0, \infty))>0$, then $X$ is a mixed diffusion/jump process.

By \cite[Theorem 2.1, formula (2.8)]{Oku}, the jump kernel of a SBM is
\begin{equation*}
    J(x, y) = \int_{(0, \infty)} (2\pi s)^{-d/2} \exp \left( - \frac{|x-y|^2}{2s} \right) \, \mu(ds).
\end{equation*}
As usual, let $j(r)$ denote the common value of $J(x, y)$ for all $x, y \in \mathbb{R}^d$ such that $|x-y|=r$:
\begin{equation} \label{JumpKernel}
    j(r) = \int_{(0, \infty)} (2\pi s)^{-d/2} \exp \left( - \frac{r^2}{2s} \right) \, \mu(ds).
\end{equation}
Note also that the jump kernel is non-increasing: $j(R) \leq j(r)$ whenever $R \geq r>0$.

Therefore, by the $(3)\Longrightarrow(1)$ implication of \cite[Proposition on page 488]{W}, every SBM is an isotropic unimodal L\'{e}vy process.

\subsection{Main result}

We are now ready to state our main result.

\begin{theorem} \label{MainResult}
Fix constants $c$ and $\alpha$ such that $0<\sqrt{29}\alpha \leq c < 1$.
Let $X=(X_t)_{t \geq 0}$ be an isotropic unimodal L\'{e}vy process, and suppose there exists a sequence $(R_n) \subseteq (0, \infty)$ such that
\begin{equation} \label{Thm1Conditions}
    \frac{j(R_n)}{j(cR_n)} \to 0 \qquad\mbox{and} \qquad \mathbb{P}_0 \left( X_{\tau_{B(0, \alpha R_n)}} \in B(0, 10 R_n) \right) \to 0.
\end{equation}
Then $X$ does not satisfy $\EHI$.
\end{theorem}

The key ingredient for the proof of Theorem \ref{MainResult} is the following proposition, which is proved in Section \ref{DiagramSection}.

\begin{prop} \label{DiagramProp}
Fix some constants $c$ and $\alpha$ such that
$0<\sqrt{29}\alpha \leq c < 1$.
Let $X=(X_t)_{t \geq 0}$ be an isotropic unimodal L\'{e}vy process.
Let $R>0$. Let $\mathcal{A}$ and $\mathcal{U}$ be the cylinders
\begin{equation*}
    \mathcal{A} := (-R, R) \times B^{(d-1)}(0, R), \qquad \mathcal{U} := \Big((1+\alpha)R, (3+\alpha)R\Big) \times B^{(d-1)}(0, R).
\end{equation*}
Let $f$ be the function
\begin{equation*}
    f(x) := \mathbb{P}_x(X_{\tau_{\mathcal{A}}} \in \mathcal{U}).
\end{equation*}
Let $x_0 := (1-\alpha) R e_1$, where $e_1$ is the unit vector $(1, 0, 0, \dots, 0)$.

Then
\begin{equation}\label{DiagramPropResult}
    f(-x_0) \leq \left( \alpha^{-d} \frac{j(R)}{j(cR)} + 2 \mathbb{P}_0 \left( X_{\tau_{B(0, \alpha R)}} \in B(0, 10 R ) \right) \right) \sup_{B(0, (1-\alpha/2) R)} f.
\end{equation}
\end{prop}

Figure \ref{DiagramPropSetupSketch} contains a sketch of the set-up of Proposition \ref{DiagramProp}. For simplicity, we use $d=2$ in our sketch, so the ``cylinders" are simply squares.

\begin{figure}
\begin{framed}
\captionof{figure}{Sketch of the setting of Proposition \ref{DiagramProp}}
\label{DiagramPropSetupSketch}
\centering
\begin{tikzpicture}
    \draw[color=blue] (-3, -3) -- (3, -3) -- (3, 3) -- (-3, 3) -- (-3, -3);
    
    \draw[color=blue] (3.5, -3) -- (9.5, -3) -- (9.5, 3) -- (3.5, 3) -- (3.5, -3);
    
    \node[shape=circle, fill=black, scale=0.1, label=$0$] (0) at (0, 0) {$0$};
    
    \node[shape=circle, fill=black, scale=0.1, label={$-x_0$}] at (-2.5, 0) {$-x_0$};
    
    \node[shape=circle, fill=black, scale=0.1, label={$x_0$}] (x0) at (2.5, 0) {$x_0$};
    
    \node[color=blue] at (0, -3.3) {$\mathcal{A}$};
    
    \node[color=blue] at (6.5, -3.3) {$\mathcal{U}$};
    
    \node[shape=circle, fill=blue, scale=0.1] (1) at (3, 0) {};
    \node[shape=circle, fill=blue, scale=0.1] (1+a) at (3.5, 0) {};
    \node[shape=circle, fill=blue, scale=0.1] (3+a) at (9.5, 0) {};
    \node[shape=circle, fill=blue, scale=0.1] (1') at (3, -3) {};
    \node[shape=circle, fill=blue, scale=0.1] (1+a') at (3.5, -3) {};
    \node[shape=circle, fill=white, scale=0.1] (top) at (-3.5, 3) {};
    \node[shape=circle, fill=white, scale=0.1] (bottom) at (-3.5, -3) {};
    
    \path[<->, color=red] (0) edge node[below] {$(1-\alpha) R$} (x0);
    
    \path[<->, color=cyan] (x0) edge node[below] {$\alpha R$} (1);
    
    \path[<->, color=red] (1') edge node[below] {$\alpha R$} (1+a');
    
    \path[<->, color=cyan] (1+a) edge node[below] {$2 R$} (3+a);
    
    \path[<->, color=cyan] (top) edge node[left] {$2R$} (bottom);
\end{tikzpicture}
\end{framed}
\end{figure}
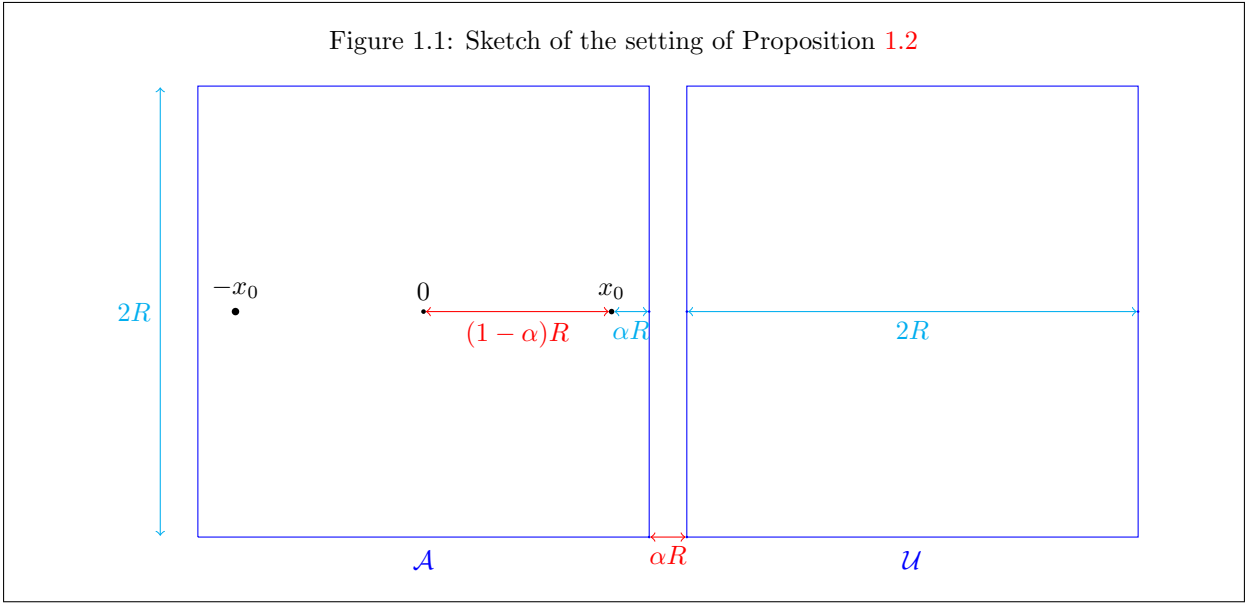

The proof of Theorem \ref{MainResult} from Proposition \ref{DiagramProp} is simple.

\begin{proof}[Proof of Theorem \ref{MainResult}]
Assume for the sake of contradiction that $X$ satisfies $\EHI$. Let $C$ be the constant from \eqref{HarnackDef}, for $\kappa=1-\alpha/2$. Whenever a function $h$ is non-negative and harmonic on a ball $B(x_0, r)$, we have
\begin{equation} \label{AssumeHarnack}
    h(x) \geq C^{-1} h(y) \qquad\mbox{for all} \quad x, y \in B \left( x_0, \left( 1 - \frac{\alpha}{2} \right) r \right).
\end{equation}
Let $(R_n)$ be a sequence satisfying \eqref{Thm1Conditions}. For each $n$, let $\mathcal{A}_n$ and $\mathcal{U}_n$ be the cylinders
\begin{equation*}
    \mathcal{A}_n := (-R_n, R_n) \times B^{(d-1)}(0, R_n), \qquad \mathcal{U}_n := \left((1+\alpha)R_n, (3+\alpha)R_n \right) \times B^{(d-1)}(0, R_n).
\end{equation*}
Let $x_n := (1-\alpha) R_n e_1$, and let $f_n$ be the function
\begin{equation*}
    f_n(x) := \mathbb{P}_x \left(X_{\tau_{\mathcal{A}_n}} \in \mathcal{U}_n \right).
\end{equation*}
(These are the $\mathcal{A}$, $\mathcal{U}$, $x_0$, and $f$ from Proposition \ref{DiagramProp} when $R=R_n$.)

We claim that $f_n$ is harmonic on $\mathcal{A}_n$. Let $E$ be an open set whose closure is compact and contained in $\mathcal{A}_n$. Fix $x \in E$. We would like to show that $f_n(x) = \mathbb{E}_x \left[ f_n(X_{\tau_E}) \right]$. Let $(\mathcal{F}_t)_{t\geq0}$ be the minimal complete, right-continuous filtration of $\sigma$-fields such that $X$ is adapted to $(\mathcal{F}_t)$. Then $\tau_E$ is an $(\mathcal{F}_t)$-stopping time, so
\begin{align*}
    f_n(x_n) &= \mathbb{P}_x \left(X_{\tau_{\mathcal{A}_n}} \in \mathcal{U}_n \right) \\
    &= \mathbb{E}_x \left[ \mathbb{P}_{\tau_E} \left(X_{\tau_{\mathcal{A}_n}} \in \mathcal{U}_n \right) \right] \quad\mbox{(by the Strong Markov property)}\\
    &= \mathbb{E}_x \left[ f_n \left(X_{\tau_E} \right) \right].
\end{align*}

By Proposition \ref{DiagramProp} and \eqref{Thm1Conditions},
\begin{equation*}
    f_n(-x_n) \leq o(1) \cdot \sup_{B(0, (1-\alpha/2) R)} f.
\end{equation*}
Thus, for sufficiently large $n$,
\begin{equation*}
    f_n(-x_n) < \frac{C^{-1}}{2} \sup_{B(0, (1-\alpha/2) R)} f.
\end{equation*}
However, since $-x_n \in B(0, (1-\alpha/2)R_n)$, and $f_n$ is harmonic on $B(0, R_n)$, \eqref{AssumeHarnack} guarantees that
\begin{equation*}
    f_n(-x_n) \geq \frac{C^{-1}}{2} \sup_{B(0, (1-\alpha/2) R)} f.
\end{equation*}
This is a contradiction. Thus, $X$ does not satisfy $\EHI$.
\end{proof}

The rest of the paper is organized as follows. In Section \ref{PreliminaryResultsSection}, we discuss some basic preliminary results. Among these is the L\'{e}vy system formula, a tool used in \cite[Section 3]{BC}, which we will make similar use of.
In Section \ref{PreferredSideSection}, we prove a technical lemma, involving a $(d-1)$-dimensional hyperplane that splits the space in two, and which side of the hyperplane the process is on.
In Section \ref{DiagramSection}, we prove Proposition \ref{DiagramProp}.
In Section \ref{CounterexamplesSection}, we construct some specific subordinated Brownian motions which violate $\EHI$, by providing examples that satisfy the condition \eqref{Thm1Conditions} in Theorem \ref{MainResult}.

\subsection{Acknowledgements}

We would like to thank Tomasz Grzywny for telling us about his example in \cite{GK}, and Stephen Zhang for his helpful technical advice on how to generate Figure \ref{JRatioGraphs}.

\section{Preliminary results} \label{PreliminaryResultsSection}

\subsection{L\'{e}vy system formula}

The first key ingredient, in both our proof of Proposition \ref{DiagramProp}, and Bass and Chen's counterexample in \cite[Section 3]{BC}, is the \textit{L\'{e}vy system formula}. More general forms of the L\'{e}vy system formula exist, but the following form will be sufficient for our purposes.

\begin{lemma} \label{LevySystemFormua}
Let $X=(X_t)_{t \geq 0}$ be a L\'{e}vy process on $\mathbb{R}^d$, adapted to a filtration $(\mathcal{F}_t)_{t \geq 0}$, and suppose $X$ has a jump kernel $J(x, y)$. Let $f$ be a non-negative measurable function on $\mathbb{R}^d \times \mathbb{R}^d$ that vanishes on the diagonal. Then, for any $x \in \mathbb{R}^d$ and any $(\mathcal{F}_t)$-stopping time $T$,
\begin{equation*}
    \mathbb{E}_x \left[ \sum_{s \leq T} f(X_{s-}, X_s) \right] = \mathbb{E}_x \left[ \int_0^T \int_{\mathbb{R}^d} f(X_s, y) J(X_t, y) \, dy \, ds \right].
\end{equation*}
\end{lemma}

Note that $X$ does not need to be a pure-jump process for Lemma \ref{LevySystemFormua} to apply. There can be both a diffusion component and a jump component. For a proof in a much more general setting, see \cite[Theorem 4.3.3(ii)]{CF}.

In particular, if $\mathcal{A}$ and $\mathcal{U}$ are disjoint open sets, $f$ is the function $f(x, y) = \indicatorWithSetBrackets{x \in \mathcal{A}, y \in \mathcal{U}}$, and $T$ is the stopping time $\tau_{\mathcal{A}}$ (as defined in \eqref{ExitTimeDef}), then the L\'{e}vy system formula gives us
\begin{equation} \label{LsfApplied}
    \mathbb{P}_x \left( X_{\tau_{\mathcal{A}}} \in \mathcal{U} \right) = \mathbb{E}_x \left[ \int_0^{\tau_{\mathcal{A}}} J(X_s, \mathcal{U}) \, ds \right] \qquad\mbox{for all $x \in \mathcal{A}$}.
\end{equation}

Equation \ref{LsfApplied} will be used in several places in the proof of Proposition \ref{DiagramProp}, in which one of the primary objects is a function of the form $f(x) = \mathbb{P}_x \left( X_{\tau_{\mathcal{A}}} \in \mathcal{U} \right)$ on $\mathcal{A}$. 

\subsection{A lemma concerning jumps of intermediate size}

\FloatBarrier

The following observation will also be helpful in the proof of Proposition \ref{DiagramProp}.

\begin{figure}
\begin{framed}
\captionof{figure}{Sketch of the setting of Lemma \ref{IntermediateJumpE1E2}}
\label{IntermediateJumpE1E2Figure}
\centering
\begin{tikzpicture}[scale=0.5]
    \node[shape=circle, fill=black, scale=0.1, label=$x$] (0) at (0, 0) {$x$};

    \draw[color=blue] (-1, 3) -- (1, 3) -- (1, -3) -- (-1, -3) -- (-1, 3);
    \node[color=blue] at (0, 3.3) {$E_1$};
    
    \draw[color=blue] (3, 3) -- (5, 3) -- (5, -3) -- (3, -3) -- (3, 3);
    \node[color=blue] at (4, 3.3) {$E_2$};
    
    \draw[color=red, fill=none] (0, 0) circle (2);
    \node[color=red] at (-3.4, 0) {$B(x, r_1)$};
    
    \draw[color=red, fill=none] (0, 0) circle (6);
    \node[color=red] at (3.6, 6) {$B(x, r_2)$};
\end{tikzpicture}
\end{framed}
\end{figure}
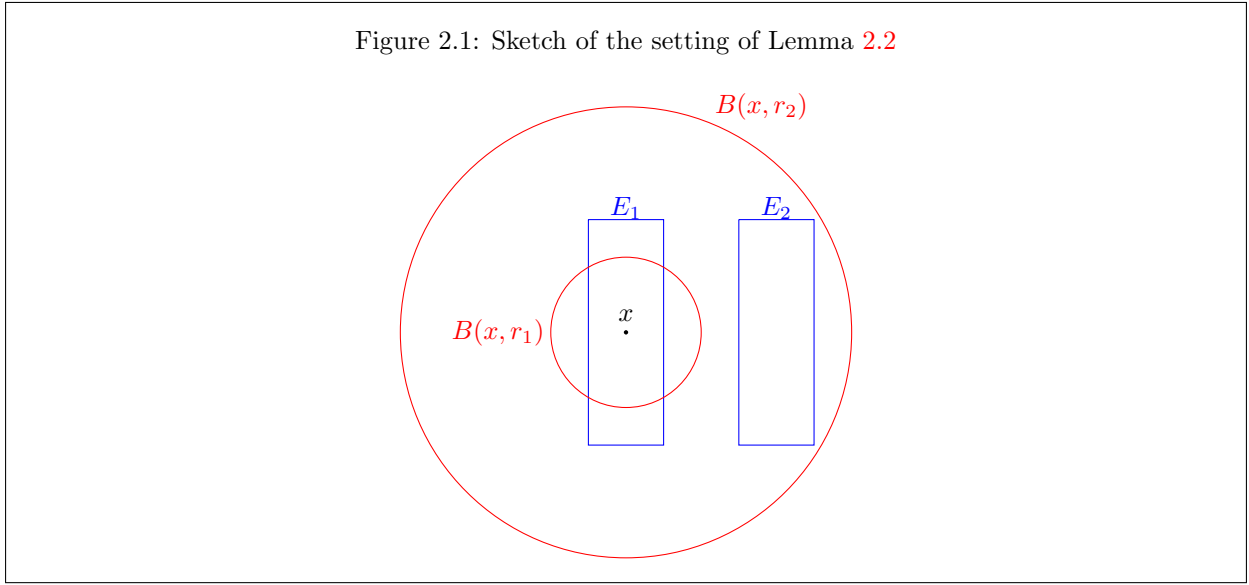

Consider the scenario shown in Figure \ref{IntermediateJumpE1E2Figure}.
There are two disjoint sets $E_1$ and $E_2$. We consider two concentric balls $B(x, r_1)$ and $B(x, r_2)$, both centered at a point $x \in E_1$. The smaller ball is small enough that it does not intersect $E_2$, while the larger ball is so large that it contains all of $E_1$ and $E_2$.

Suppose we are interested in the probability $\mathbb{P}_x \left( X_{\tau_{E_1}} \in E_2 \right)$. Recall from \eqref{ExitTimeDef} that $X_{\tau_{E_1}}$ is the first point outside of $E_1$ that the process $X$ reaches. As we will show, there is no way for $X_{\tau_{E_1}}$ to be in $E_2$ without $X_{\tau_{B(x, r_1))}}$ being in $B(x, r_2)$. Therefore,
\begin{equation*}
    \mathbb{P}_x \left( X_{\tau_{E_1}} \in E_2 \right) \leq \mathbb{P}_x \left( X_{\tau_{B(x, r_1)}} \in B(x, r_2) \right).
\end{equation*}
This allows us to replace a probability involving potentially complicated sets $E_1$ and $E_2$ with one involving nice open balls. This will be especially helpful in situations such as Theorem \ref{MainResult}, in which the location of $X$ at the exit time of a small ball $B(0, \alpha R_n)$ is unlikely to be inside even the much larger ball $B(0, 10 R_n)$.

This observation is turned into a proof in the following lemma. Because $X$ is translation-invariant, we may also recenter the balls at the origin.

\begin{lemma} \label{IntermediateJumpE1E2}
Let $X=(X_t)_{t \geq 0}$ be a L\'{e}vy process on a metric measure space. Suppose there exist two disjoint open sets $E_1$ and $E_2$, a point $x \in E_1$, and positive numbers $r_2>r_1>0$ such that
\begin{itemize}
    \item The ball $B(x, r_1)$ does not intersect $E_2$.
    \item The ball $B(x, r_2)$ contains both $E_1$ and $E_2$.
\end{itemize}
Then
\begin{equation*}
    \mathbb{P}_{x} \left( X_{\tau_{E_1}} \in E_2 \right) \leq \mathbb{P}_0 \left( X_{\tau_{B(0, r_1)}} \in B(0, r_2) \right).
\end{equation*}
\end{lemma}

\begin{proof}
First, we will show that
\begin{equation} \label{EventsInsteadOfProbabilities}
    \left\{ X_0=x \right\} \cap \left\{ X_{\tau_{E_1}} \in E_2 \right\} \subseteq \left\{ X_{\tau_{B(x, r_1)}} \in B(x, r_2) \right\}.
\end{equation}
Suppose the process begins at $x$, and $X_{\tau_{E_1}}$ is in $E_2$. There are two ways that this can happen: either the process exits
$B(x, r_1)$ before it exits $E_1$, and then enters $E_2$ at the moment if first exits $E_1$; or the process leaves $B(x, r_1)$ and $E_1$ at the same time, and enters $E_2$ at the moment it does so. 
In the first case, $X_{\tau_{B(x, r_1)}} \in E_1$. In the second case, $X_{\tau_{B(x, r_1)}} \in E_2$. Since both $E_1$ and $E_2$ are contained in $B(x, r_2)$, we have $X_{\tau_{B(x, r_1)}} \in B(x, r_2)$ either way. This completes the proof of \eqref{EventsInsteadOfProbabilities}.

By taking probabilities of these events, and then applying translation-invariance to recenter,
\begin{equation*}
    \mathbb{P}_x \left( X_{\tau_{E_1}} \in E_2 \right) \leq \mathbb{P}_x \left( X_{\tau_{B(x, r_1)}} \in B(x, r_2) \right) = \mathbb{P}_0 \left( X_{\tau_{B(0, r_1)}} \in B(0, r_2) \right).
\end{equation*}
\end{proof}

\FloatBarrier

\section{``Preferred side" lemmas} \label{PreferredSideSection}

In this section, we prove two technical results, which we call the ``Preferred side" lemmas (for reasons that will be made clear). One of these will be used in the proof of Proposition \ref{DiagramProp}.

Let us first describe the setting. Let $H$ be a $(d-1)$-dimensional subspace of $\mathbb{R}^d$. Note that $H$ divides $\mathbb{R}^d$ into two half-spaces; let us call these half-spaces $V$ and $W$. (For simplicity, we can take $V$ to be closed and $W$ to be open, so the two half-spaces partition the full space, and $H \subseteq V$.)

Fix $H$, $V$, and $W$ for the remander of this section. Let us call $V$ the ``preferred side." The results of this section will concern situations in which $V$ is ``favored" over $W$ in some way: the process begins in $V$, a certain function takes larger values in $V$ than in $W$, or a certain set has a larger intersection with $V$ than with $W$. From these asymmetries, we derive results that are intuitive, but technically difficult to prove.

First, let us introduce some notation.
For all $x \in \mathbb{R}^d$, let $x'$ denote the reflection of $x$ across the hyperplane $H$. Also, let
\begin{equation} \label{VarphiDef}
    \varphi(x) := \left\{ \begin{matrix}
        x &:& \mbox{if $x \in V$}\\
        x' &:& \mbox{if $x \in W$}.
    \end{matrix}\right.
\end{equation}
In other words, $\varphi(x)$ is equal to whichever one of $x$ or $x'$ is in the preferred side $V$.

Consider the process $\varphi(X) = (\varphi(X_t))_{t \geq 0}$. Let $(\mathcal{G}_t)_{t \geq 0}$ be the smallest complete, right-continuous filtration that $\varphi(X)$ is adapted to. This process behaves somewhat analogously to reflected Brownian motion, which always stays on one side of a boundary (the difference being that $\varphi(X)$ can have jumps in addition to diffusion).

If we look only at the history of $\varphi(X)$, we lose track of which side of $H$ the original process $X$ is on.
However, the following lemma tells us that if $X$ begins on the preferred side, and we condition on the history of $\varphi(X)$, then for any $(\mathcal{G}_t)$-stopping time $T$, $X_T$ is more likely than not to be on the preferred side.

\begin{lemma} \label{PreferredSideZ}
Let $X=(X_t)_{t \geq 0}$ be an isotropic unimodal L\'{e}vy process on $\mathbb{R}^d$.
Let $H$ be a $(d-1)$-dimensional hyperplane in $\mathbb{R}^d$, and let $V$ and $W$ be the two half-spaces that $H$ divides $\mathbb{R}^d$ into, with $V$ closed and $W$ open.
For all $x \in \mathbb{R}^d$, let $x'$ denote the reflection of $x$ across $H$.
Let $\varphi(X) = (\varphi(X_t))_{t \geq 0}$, where $\varphi$ is as defined in \eqref{VarphiDef}.
Let $(\mathcal{G}_t)_{t \geq 0}$ be the minimal complete, right-continuous filtration such that $\varphi(X)$ is adapted to $(\mathcal{G}_t)$.

Then for any $v \in V$ and $t \geq 0$,
\begin{equation*}
    \mathbb{P}_v \left( X_t \in V \left| \mathcal{G}_t \right.\right) \geq \frac12 \qquad\mbox{almost surely}.
\end{equation*}

Furthermore, let $E$ be a pre-compact subset of $V$, which is open in the relative topology on $V$.
Let $\tau^{\varphi(X)}_E:=\inf\{ t: \varphi(X_t) \notin E\}$.
Then for any $v \in V$,
\begin{equation*}
    \mathbb{P}_v \left( X_{\tau^{\varphi(X)}_E} \in V \left| \mathcal{G}_{\tau^{\varphi(X)}_E} \right. \right) \geq \frac12 \qquad\mbox{almost surely}.
\end{equation*}
\end{lemma}

We do not prove Lemma \ref{PreferredSideZ} just yet, since it is necessary to develop some more machinery first, but let us briefly summarize how the proof will unfold. We will construct a joint process $(Y_t, Z_t)_{t \geq 0}$, which shares the law of $(X_t, \varphi(X_t))_{t \geq 0}$, but where $Z$ is constructed first. Then, in order to determine $Y$, we must choose at every jump in $Z$ whether $Y_t$ jumps across $H$ or not; we determine this using a random structure independent from $Z$, reverse-engineered so that the law of $Y$ matches that of $X$. Then we can use the random structure independent of $Z$ to calculate $\mathbb{P}_v \left( X_T \in V \left| \mathcal{G} \right. \right)$.

We will then use Lemma \ref{PreferredSideZ} to prove the following result. To put it imprecisely, this result says that if there exists an open set $D$ which has a larger intersection with the preferred side than the non-preferred side, and a non-negative function $g$ on $D$ which takes larger values on the preferred side, then $\int_0^{\tau_D} g(X_t) \, dt$ should be larger for initial values on the preferred side than on the non-preferred side.

\begin{lemma} \label{PreferredSideW}
Let $X=(X_t)_{t \geq 0}$ be an isotropic unimodal L\'{e}vy process on $\mathbb{R}^d$.
Let $H$ be a $(d-1)$-dimensional hyperplane in $\mathbb{R}^d$, and let $V$ and $W$ be the two half-spaces that $H$ divides $\mathbb{R}^d$ into, with $V$ closed and $W$ open.
For all $x \in \mathbb{R}^d$, let $x'$ denote the reflection of $x$ across $H$, and
let $\varphi(x)$ be as defined in \eqref{VarphiDef}.

Let $D$ be an open subset of $\mathbb{R}^d$, such that
for all $w \in W$,
\begin{equation} \label{DUlargerThanDV}
    w \in D \Longrightarrow w' \in D.
\end{equation}
Let $g:D \to [0, \infty)$ be a measurable function such that for all $w \in W$,
\begin{equation} \label{GLargerOnU}
    w \in D \Longrightarrow g(w) \leq g(w').
\end{equation}

Then, for all $w \in D \cap W$,
\begin{equation*}
    \mathbb{E}_{w} \left[ \int_0^{\tau_D} g(X_t) \, dt \right] \leq \mathbb{E}_{w'} \left[ \int_0^{\tau_D} g(X_t) \, dt \right].
\end{equation*}
\end{lemma}

Note that Lemma \ref{PreferredSideW} will be directly used outside this section. Lemma \ref{PreferredSideZ} is only used to prove Lemma \ref{PreferredSideW}.

We start with the following helpful fact.

\begin{lemma} \label{BernoulliSumEven}
Let $(a_i)_{i \in I}$ be a sequence, indexed by some finite or countable sequence $I$, such that $0 \leq a_i \leq 1$ for all $i$, and $\sum_{i \in I} a_i < \infty$.
Let $(\eta_i)_{i \in I}$ be a collection of independent random variables, also indexed by $I$, such that each $\eta_i$ is Bernoulli($a_i$).
Then
    \begin{equation} \label{BernoulliSumEvenConclusion}
        \mathbb{P}\left( \mbox{$\sum_{i \in I} \eta_i$ is even} \right) = \frac12 \left( 1 + \prod_{i \in I} (1-2a_i) \right).
    \end{equation}
\end{lemma}

\begin{proof} Since the $a_i$'s have a finite sum, by Borel-Cantelli, $\sum_{i \in I} \eta_i$ is almost surely finite. By independence,
\begin{equation*}
    \mathbb{E}\left[ (-1)^{\sum_{i \in I} \eta_i} \right] = \prod_{i \in I} \mathbb{E}\left[ (-1)^{\eta_i} \right].
\end{equation*}
Simplifying,
\begin{equation*}
    2 \mathbb{P}\left( \mbox{$\sum_{i \in I} \eta_i$ is even} \right) - 1 = \prod_{i \in I} (1-2a_i).
\end{equation*}
Solving for $\mathbb{P}\left( \mbox{$\sum_{i \in I} \eta_i$ is even} \right)$, we obtain \eqref{BernoulliSumEvenConclusion}.
\end{proof}

Recall that our plan for the proof of Lemma \ref{PreferredSideZ} is to construct a joint process $(Y_t, Z_t)_{t \geq 0}$, with the same law as $(X_t, \varphi(X_t))_{t \geq 0}$, where $Z$ is constructed first, and then $Y$ is determined from $Z$ and another random structure independent of $Z$. The following definition and remark help us with the construction of $(Y_t, Z_t)_{t \geq 0}$.

\begin{definition} \label{FlipDef}
Let us say that a \textit{flip} of $X$ occurs at time $t$ if $X_{t-} \in V$ and $X_t \in W$, or vice versa.
\end{definition}

\begin{remark} \label{FlipRemark}
(a) The jump kernel of the process $\varphi(X) = \left(\varphi(X_t)\right)_{t \geq 0}$ is
\begin{equation*}
    J_{\varphi(X)}(x, y) = J(x, y) + J(x, y') \qquad\mbox{for all $x, y \in V$, $x \neq y$}.
\end{equation*}
To see this, note that jumps in $\varphi(X)$ from $x$ to $y$ occur precisely whenever $X$ takes a jump from either $x$ or $x'$ to either $y$ or $y'$. Jumps from $x$ to $y$ (or $x'$ to $y'$) occur with rate $J(x, y)$, and jumps from $x$ to $y'$ (or $x'$ to $y$) occur with rate $J(x, y')$.

(b) Suppose $\varphi(X)$ takes a jump from $x$ to $y$ at time $t$. Then $X$ must take a jump from either $x$ or $x'$ to either $y$ or $y'$. The probability that this jump is a flip in $X$ is equal to $J(x, y') / \left( J(x, y) + J(x, y') \right)$.
This quantity will come up in the proof of Lemma \ref{PreferredSideZ}, so let us simplify notation by setting
\begin{equation} \label{pFlipDef}
    \pFlip(x, y) := \frac{J(x, y')}{J(x, y)+J(x, y')} \qquad\mbox{for all $x, y \in V$, $x \neq y$}.
\end{equation}
\end{remark}
Note that for any $x, y \in V$, the distance between $x$ and $y'$ is greater than or equal to the distance between $x$ and $y'$. Since $j$ is decreasing, this means that
\begin{equation*}
    J(x, y') = j(|x-y'|) \leq j(|x-y|) = J(x, y),
\end{equation*}
so
\begin{equation} \label{pFlipLeqHalf}
    \pFlip(x, y) \leq \frac12 \qquad\mbox{for all $x, y \in V$, $x \neq y$}.
\end{equation}

\begin{proof}[Proof of Lemma \ref{PreferredSideZ}]
Let $Z=(Z_t)_{t \geq 0}$ be a stochastic process on $V$ with the same law as $\varphi(X)=(\varphi(X_t))_{t\geq0}$.
Let $(\mathcal{H}_t)_{t \geq 0}$ be the minimal complete, right-continuous filtration such that $Z$ is adapted to $(\mathcal{H}_t)$.

Let $(U_t)_{t \in (0, \infty)}$ be an independent, identically-distributed collection of random variables, uniform on $(0, 1)$, indexed by $[0, \infty)$, such that the whole collection $(U_t)$ is independent from $Z$.

In order to determine $Y$ from $Z$, we must determine which jumps in $Z$ are flips in $Y$. Whenever $Z$ takes a jump from $x$ to $y$, the probability of a flip should be equal to $\pFlip(x, y)$ (as defined in \eqref{pFlipDef}).
Therefore, let
\begin{equation*}
    F = \left\{ t>0 : Z_{t-}\neq Z_t, U_t \leq \pFlip(Z_{t-}, Z_t) \right\}.
\end{equation*}
The random set $F$ will be the set of times at which a flip in $Y$ occurs.

There may also be times $t$ such that $Z_t \in H$. In such cases, we must also determine which side of $H$ the process $Y$ goes into after time $t$. In every such a case, we should choose either $V$ or $W$, each with probability $1/2$. This is achieved in the following construction.

Fix $t \geq 0$. Suppose first that there does not exist any $s \in [0, t]$ such that $Z_s \in H$. In this case, let
\begin{equation*}
    Y_t := \left\{ \begin{matrix}
        Z_t &:& \mbox{if $\#(F \cap (0, t])$ is even}\\
        \\
        Z_t' &:& \mbox{if $\#(F \cap (0, t])$ is odd}.
    \end{matrix}\right.
\end{equation*}
Otherwise, let $s$ be the maximal time less than or equal to $t$ such that $Z_s \in H$ or $Z_{s-} \in H$, and let
\begin{equation*}
    Y_t := \left\{ \begin{matrix}
        Z_t &:& \mbox{if $U_s \leq \frac12$ and $\#(F \cap (s, t])$ is even, or if $U_s > \frac12$ and $\#(F \cap (s, t])$ is odd}\\
        \\
        Z_t' &:& \mbox{if $U_s \leq \frac12$ and $\#(F \cap (s, t])$ is odd, or if $U_s > \frac12$ and $\#(F \cap (s, t])$ is even}.
    \end{matrix}\right.
\end{equation*}
Then $(Y_t, Z_t)_{t \geq 0}$ has the same law as $(X_t, \varphi(X_t))_{t \geq 0}$.

Recall that we defined $(\mathcal{G}_t)_{t \geq 0}$ to be the minimal complete, right-continuous filtration such that $\varphi(X)$ is adapted to $(\mathcal{G}_t)$. By analogy, let $(\mathcal{H}_t)_{t\geq0}$ be the minimal complete, right-continuous filtration such that $Z$ is adapted to $(\mathcal{H}_t)$.

Fix $t \geq 0$.
Let us condition on the history of $Z$ from time $0$ to time $t$. If for any time $s$ in this period, we have $Z_s \in H$ or $Z_{s-} \in H$, then the conditional probability that $Y_{t} \in V$ is at least $1/2$, by symmetry. (We say ``at least $1/2$" rather than ``exactly $1/2$" because the boundary $H$ belongs to $V$). If there is no such $s$, then $Y_{t} \in V$ as long as the number of flips before time $t$ is even. The potential flips occur independently, each with its own probability that does not exceed $1/2$, so by Lemma \ref{BernoulliSumEven}, the conditional probability that the number of flips is even is at least $1/2$. Therefore, for all $v \in V$,
\begin{equation} \label{DeterministicTime31}
    \mathbb{P}_v \left(Y_t \Big| \mathcal{H}_{t} \right) \geq \frac12. 
\end{equation}
Since $(Y_t, Z_t)_{t \geq 0}$ has the same law as $(X_t, \varphi(X_t))_{t \geq 0}$,
\begin{equation*}
    \mathbb{P}_v \left(X_t \Big| \mathcal{G}_t \right)=\mathbb{P}_v \left(Y_t \Big| \mathcal{H}_t \right) \geq \frac12.
\end{equation*}

Let $E$ be a pre-compact subset of $V$, which is open in the relative topology on $V$.
Let $\tau^{\varphi(X)}_E:=\inf\{ t: \varphi(X_t) \notin E\}$ and $\tau^Z_E := \inf\{t:Z_t \notin E\}$.
Since $E$ is open, $\tau^Z_E$ is a $(\mathcal{H}_t)$-stopping time. For all $v \in V$, by the same argument that we used to obtain \eqref{DeterministicTime31}, applied to $\tau^Z_E$ instead of $t$, we obtain
\begin{equation*}
    \mathbb{P}_v \left(Y_{\tau^Z_E} \Big| \mathcal{H}_{\tau^Z_E} \right) \geq \frac12. 
\end{equation*}
Since $(Y_t, Z_t)_{t \geq 0}$ has the same law as $(X_t, \varphi(X_t))_{t \geq 0}$,
\begin{equation*}
    \mathbb{P}_v \left(X_{\tau^{\varphi(X)}_E} \Big| \mathcal{G}_{\tau^{\varphi(X)}_E} \right)=\mathbb{P}_v \left(Y_{\tau^Z_E} \Big| \mathcal{H}_{\tau^Z_E} \right) \geq \frac12.
\end{equation*}
\end{proof}

\begin{proof}[Proof of Lemma \ref{PreferredSideW}]

Let $D$ and $g$ satisfy \eqref{DUlargerThanDV} and \eqref{GLargerOnU}.
Let $G : D \to [0, \infty)$ be the function
\begin{equation*}
    G(x) := \mathbb{E}_x\left[ \int_0^{\tau_D} g(X_t) \, dt \right] \qquad\mbox{for all $x \in D$}.
\end{equation*}

Fix $w \in D \cap W$. We would like to show that $G(w) \leq G(w')$. Let $Y=(Y_t)_{t \geq 0}$ be a process with the same law as $X$, starting at $Y_0=w$. We will also consider the process $Y'=(Y'_t)_{t \geq 0}$, where $Y'_t$ is the reflection of $Y_t$ across $H$. This reflected process also has the same law as $X$, but its initial value is $Y'_0=w' \in V$.

Let $Z_t := \varphi(Y_t) = \varphi(Y'_t)$ for all $t$.
Let $(\mathcal{G}_t)_{t \geq 0}$ be the minimal complete, right-continuous filtration such that $Z=(Z_t)_{t \geq 0}$ is adapted to $(\mathcal{G}_t)$.

Note that for all $t$, either $Y_t \in V$ or $Y'_t \in V$.

Let
\begin{align*}
    \tau^Y_D &:= \inf\{ t\geq0 : Y_t \notin D\}, \\
    \tau^{Y'}_D &:= \inf\{ t\geq0 : Y'_t \notin D\}, \\
    T &:= \min\{\tau^Y_D, \tau^{Y'}_D\}, \\
    \tilde{T} &:= \max\{\tau^Y_D, \tau^{Y'}_D\}.
\end{align*}

Let
\begin{equation*}
    D_0 := (D \cap W) \cup \{w' : w \in D \cap W\}.
\end{equation*}
In other words, $D_0$ is the largest subset of $D$ that is symmetric about $H$. Let
\begin{equation*}
    \tilde{D} := D \setminus D_0 = \{x \in D : x' \notin D \}.
\end{equation*}
Note that $\tilde{D} \subseteq D \cap V$.

Observe that $T = \inf\{ t : \mbox{$Y_t \notin D$ or $Y'_t \notin D$} \} = \inf\{ t : Z_t \notin D_0 \cap V\}$. Since $D_0 \cap V$ is a pre-compact subset of $V$, and is open in the relative topology on $V$, by Lemma \ref{PreferredSideZ}, we have
\begin{equation} \label{ApplyingFirstPSLemma}
    \mathbb{P}\left( Y_T \in V \left|  \mathcal{G}_T \right.\right) \geq \frac12.
\end{equation}

In order to show that $G(w) \leq G(w')$, we will consider the quantity $G(w') - G(w)$, and show that it is non-negative. Since $Y$ and $Y'$ have the same law as $X$ and initial values $w$ and $w'$ (respectively),
\begin{align*}
    G(w') - G(w) &= \mathbb{E}_{w'}\left[ \int_0^{\tau_D} g(X_t) \, dt \right] - \mathbb{E}_w \left[ \int_0^{\tau_D} g(X_t) \, dt \right]\\
    &= \mathbb{E}\left[ \int_0^{\tau^{Y'}_D} g(Y'_t) \, dt \right] - \mathbb{E}_w \left[ \int_0^{\tau^Y_D} g(Y_t) \, dt \right].
\end{align*}
This is equivalent to
\begin{align*}
    G(w') - G(w) &= \mathbb{E}\left[\int_0^T \left[ g(Y'_t) - g(Y_t) \right] \, dt + \indicatorWithSetBrackets{\tau^{Y'}_D > \tau^Y_D} \int_T^{\tilde{T}} g(Y'_t) \, dt + \indicatorWithSetBrackets{\tau^{Y}_D > \tau^{Y'}_D} \int_T^{\tilde{T}} g(Y_t) \, dt \right].
\end{align*}
Let us separate the right-hand side into two terms, which we will handle separately:
\begin{equation*}
    G(w') - G(w) = \mathscr{A} + \mathscr{B},
\end{equation*}
where
\begin{equation*} 
    \mathscr{A}:= \mathbb{E}\left[\int_0^T \left[ g(Y'_t) - g(Y_t) \right] \, dt \right]
\end{equation*}
and
\begin{equation*} 
    \mathscr{B}:= \mathbb{E}\left[\indicatorWithSetBrackets{\tau^{Y'}_D > \tau^Y_D} \int_T^{\tilde{T}} g(Y'_t) \, dt + \indicatorWithSetBrackets{\tau^{Y}_D > \tau^{Y'}_D} \int_T^{\tilde{T}} g(Y_t) \, dt \right].
\end{equation*}

We will show that both $\mathscr{A}$ and $\mathscr{B}$ are non-negative. Let us start with $\mathscr{A}$. By Tonelli's theorem,
\begin{equation*}
    \mathscr{A} = \int_0^\infty \mathbb{E} \left[ \indicatorWithSetBrackets{T>t} \left[ g(Y'_t) - g(Y_t) \right] \right] \, dt.
\end{equation*}
By the Tower property, this becomes
\begin{align}
    \mathscr{A} &= \int_0^\infty \mathbb{E} \Bigg[ \mathbb{E} \left[ \indicatorWithSetBrackets{T>t} \left[ g(Y'_t) - g(Y_t) \right] \Big| \mathcal{G}_t \right] \Bigg] \, dt \nonumber\\
    &= \int_0^\infty \mathbb{E} \Bigg[ \mathbb{E} \left[ \indicatorWithSetBrackets{T>t} \left[\indicatorWithSetBrackets{Y'_t \in V} - \indicatorWithSetBrackets{Y'_t \notin V} \right]  \left[ g(Z_t) - g(Z'_t) \right] \Big| \mathcal{G}_t \right] \Bigg] \, dt. \label{ATower}
\end{align}
By taking all of the $\mathcal{G}_t$-measurable random variables outside of the conditional expectation, \eqref{ATower} becomes
\begin{align}
    \mathscr{A} &= \int_0^\infty \mathbb{E} \Bigg[ \indicatorWithSetBrackets{T>t} \left[ g(Z_t) - g(Z'_t) \right] \mathbb{E} \left[\indicatorWithSetBrackets{Y'_t \in V} - \indicatorWithSetBrackets{Y'_t \notin V} \bigg| \mathcal{G}_t \right] \Bigg] \, dt \nonumber\\
    &= \int_0^\infty \mathbb{E} \Bigg[ \indicatorWithSetBrackets{T>t} \left[ g(Z_t) - g(Z'_t) \right] \left(2 \mathbb{P} \left[ Y'_t \in V \Big| \mathcal{G}_t \right] - 1 \right) \, \Bigg] dt. \label{AWithXis}
\end{align}
Recall that $g$ has the property that $g(v) \leq g(v')$ for all $v \in D \cap V$. Therefore, $\indicatorWithSetBrackets{T>t} \left[ g(Z_t) - g(Z'_t) \right]$ is non-negative. Also, by Lemma \ref{PreferredSideZ},
\begin{equation*}
    2 \mathbb{P} \left[ Y'_t \in V \Big| \mathcal{G}_t \right] - 1 \geq 0.
\end{equation*}
Therefore, the integrand of \eqref{AWithXis} is non-negative. Thus, $\mathscr{A} \geq 0$.

Now let us show that $\mathscr{B} \geq 0$. Note that if $\tilde{T}>T$, then $Z_T$ must be in $\tilde{D}$. Recall that $Z_T$ is equal to $Y'_T$ if $\tau^{Y'}_D > \tau^Y_D$, or $Y_T$ if $\tau^{Y}_D > \tau^{Y'}_D$. Therefore, by the Markov property,
\begin{align*}
    \mathscr{B} &= \mathbb{E} \left[ \indicatorWithSetBrackets{\tau^{Y'}_D > \tau^Y_D} G(Y'_T) + \indicatorWithSetBrackets{\tau^{Y}_D > \tau^{Y'}_D} G(Y_T) \right]\\
    &= \mathbb{E} \left[ \indicatorWithSetBrackets{Y'_T \in V} G(Z_T) + \indicatorWithSetBrackets{Y'_T \notin V} G(Z_T) \right] \\
    &= \mathbb{E} \left[ \left( \indicatorWithSetBrackets{Y'_T \in V} - \indicatorWithSetBrackets{Y'_T \in V} \right) G(Z_T) \right].
\end{align*}
By the Tower property, this becomes,
\begin{align*}
    \mathscr{B} &= \mathbb{E} \left[ G(Z_T) \, \mathbb{E}\left[ \indicatorWithSetBrackets{Y'_T \in V} - \indicatorWithSetBrackets{Y'_T \in V} \Big| \mathcal{G}_t \right] \right]\\
    &= \mathbb{E} \left[ G(Z_T) \, \left(2\mathbb{P}\left[Y'_T \in V \Big| \mathcal{G}_t \right] - 1\right) \right].
\end{align*}
By \eqref{ApplyingFirstPSLemma}, $2\mathbb{P}\left[Y'_T \in V \Big| \mathcal{G}_t \right] - 1 \geq 0$. Since $G$ only takes non-negative values, $G(Z_T) \geq 0$. Thus, $\mathscr{B} \geq 0$.

Since $G(w')-G(w)=\mathscr{A}+\mathscr{B}$, and both $\mathscr{A}$ and $\mathscr{B}$ are non-negative, we have $G(w) \leq G(w')$, just as we wanted to prove.
\end{proof}

\FloatBarrier
\section{Proof of Proposition \ref{DiagramProp}} \label{DiagramSection}

In this section, we prove Proposition \ref{DiagramProp}. Fix constants $c$, $\alpha$, and $R$ such that $0 < \sqrt{29}\alpha \leq c < 1$ and $R>0$. For the rest of this section, let $\mathcal{A}$, $\mathcal{U}$, $f$, and $x_0$ all be precisely as they are defined in the statement of Proposition \ref{DiagramProp}. 

Recall the notation $B^{(d-1)}(x_0, r)$ from \eqref{Bd-1}.
Within $\mathcal{A}$, consider the smaller cylinders
\begin{equation*}
    A_+ := \Big((1-2\alpha)R, R \Big) \times B^{(d-1)}(0, \alpha R) \qquad\mbox{and}\qquad A_- := \Big(-R, -(1-2\alpha)R \Big) \times B^{(d-1)}(0, \alpha R).
\end{equation*}
Also, let $U$ be the smaller cylinder within $\mathcal{U}$ defined by
\begin{equation*}
    U := \Big( (1+\alpha) R, (3+\alpha) R \Big) \times B^{(d-1)}(0, \alpha R).
\end{equation*}
Finally, let
\begin{equation*}
    \mathcal{A}_\ell := \left\{ x \in \mathcal{A} : \mbox{$\inner{x, e_1} \leq 0$ and $x \notin A_-$} \right\}.
\end{equation*}
and
\begin{equation*}
    \mathcal{A}_r := \left\{ x \in \mathcal{A} : \inner{x, e_1}>0 \right\}.
\end{equation*}
See Figure \ref{Diagram} for a sketch of all these regions.

\begin{figure} 
\begin{framed}
\captionof{figure}{Sketch of the setting of Proposition \ref{DiagramProp}, including some additional named regions that will help with the proof.}

\label{Diagram}
\centering
\begin{tikzpicture}
    \draw[color=blue] (-3, -3) -- (3, -3) -- (3, 3) -- (-3, 3) -- (-3, -3);
    
    \draw[color=blue] (3.5, -3) -- (9.5, -3) -- (9.5, 3) -- (3.5, 3) -- (3.5, -3);
    
    \draw[color=red] (-2.94, -.5) -- (-2, -.5) -- (-2, .5) -- (-2.94, .5) -- (-2.94, -.5);
    
    \draw[color=red] (2.94, -.5) -- (2, -.5) -- (2, .5) -- (2.94, .5) -- (2.94, -.5);
    
    \draw[color=red] (3.53, -.5) -- (4.5, -.5) -- (4.5, .5) -- (3.53, .5) -- (3.53, -.5);
    
    \draw[color=green] (-.03, 2.94) -- (-.03, -2.94) -- (-2.94, -2.94) -- (-2.94, -.56) -- (-1.94, -.56) -- (-1.94, .56) --(-2.94, .56) -- (-2.94, 2.94) -- (-.03, 2.94);
    
    \draw[color=purple] (.03, 2.94) -- (2.97, 2.94) -- (2.97, -2.94) -- (.03, -2.94) -- (.03, 2.94);
    
    \node[shape=circle, fill=black, scale=0.1, label=$0$] at (0, 0) {$0$};
    
    \node[shape=circle, fill=black, scale=0.1, label={$-x_0$}] at (-2.5, 0) {$-x_0$};
    
    \node[shape=circle, fill=black, scale=0.1, label={$x_0$}] at (2.5, 0) {$x_0$};
    
    \node[color=red] at (-2.5, -0.4) {$A_-$};
    
    \node[color=red] at (2.5, -0.4) {$A_+$};
    
    \node[color=red] at (4, -0.4) {$U$};
    
    \node[color=green] at (-1.5, 2.7) {$\mathcal{A}_\ell$};
    
    \node[color=purple] at (1.5, 2.7) {$\mathcal{A}_r$};
    
    \node[color=blue] at (0, -3.3) {$\mathcal{A}$};
    
    \node[color=blue] at (6.5, -3.3) {$\mathcal{U}$};
\end{tikzpicture}
\end{framed}
\end{figure}
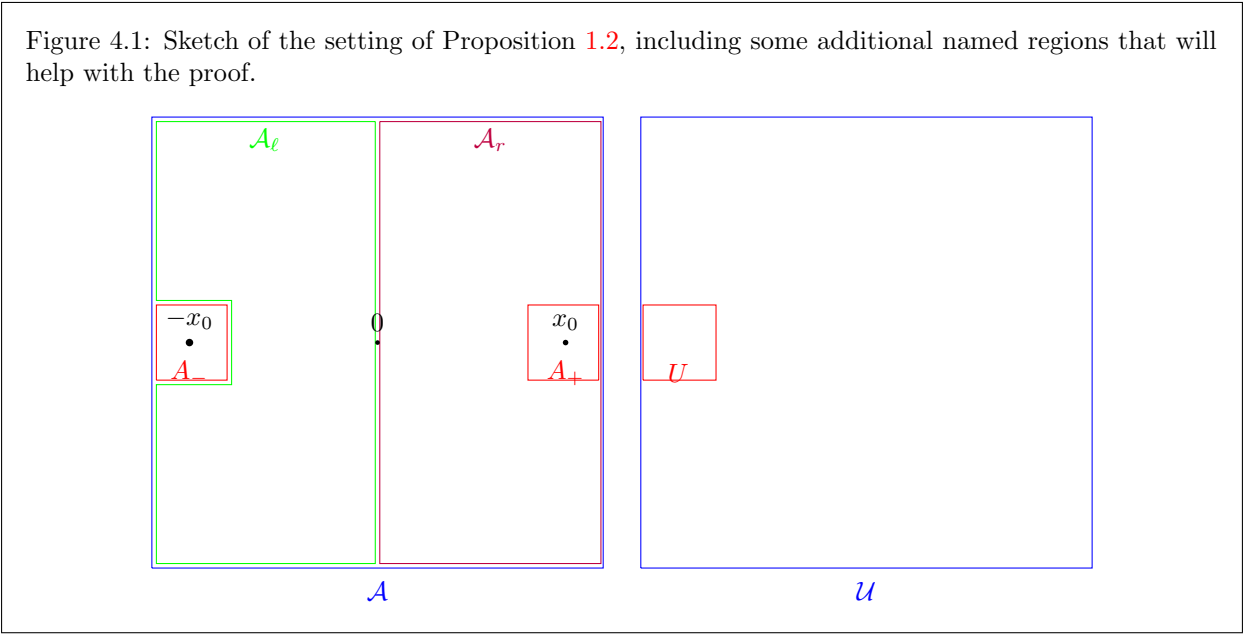

We must prove that \eqref{DiagramPropResult} holds. This entails putting an upper bound on the quantity $f(-x_0) = \mathbb{P}_{-x_0}(X_{\tau_{\mathcal{A}}} \in \mathcal{U})$.

Given a measurable set $E \subseteq \mathcal{A} \setminus A_-$, it will help to consider the quantities
\begin{align*}
    G(E) &:= \mathbb{P}_{-x_0} \left( X_{\tau_{A_-}} \in E \right),\\
    H(E) &:= \max_{y \in E} f(y) = \max_{y \in E} \mathbb{P}_y \left( X_{\tau_{\mathcal{A}}} \in \mathcal{U} \right).
\end{align*}
The quantities $G(E)$ and $H(E)$ will be useful to us because they give an upper bound on the probability, given $X_0=-x_0$, that the first point $X$ reaches outside of $A_-$ is in $E$, and the first point $X$ reaches outside of $\mathcal{A}$ is in $\mathcal{U}$.
Indeed, for all measurable $E \subseteq \mathcal{A} \setminus A_-$, by the Strong Markov property,
\begin{align}
    \mathbb{P}_{-x_0} \left( X_{\tau_{A_-}} \in E, X_{\tau_{\mathcal{A}}} \in \mathcal{U} \right) &= \mathbb{E}_{-x_0} \left[ \indicatorWithSetBrackets{X_{\tau_{A_-}} \in E} \mathbb{P}_{X_{\tau_{A_-}}} \left( X_{\tau_{\mathcal{A}}} \in \mathcal{U} \right) \right] \nonumber\\
    &= \mathbb{E}_{-x_0} \left[ \indicatorWithSetBrackets{X_{\tau_{A_-}} \in E} f \left( X_{\tau_{A_-}} \right) \right] \nonumber\\
    &\leq \mathbb{E}_{-x_0} \left[ \indicatorWithSetBrackets{X_{\tau_{A_-}} \in E} \right] \cdot \max_{y \in E} f(y) \nonumber\\
    &= \mathbb{P}_{-x_0} \left( X_{\tau_{A_-}} \in E \right) \cdot H(E)\nonumber\\
    &= G(E) H(E). \label{GEHEbound}
\end{align}

Note that $\{A_-, \mathcal{A}_\ell, \mathcal{A}_r\}$ is a partition of $\mathcal{A}$. Suppose $X_0=-x_0$, and consider the following three disjoint ways for $X_{\tau_{\mathcal{A}}}$ to be in $\mathcal{U}$:
\begin{itemize}
    \item The first point $X$ reaches outside of $A_-$ is already in $\mathcal{U}$.
    \item The first point $X$ reaches outside of $A_-$ is in $\mathcal{A}_\ell$. Later, the first point that $X$ reaches outside of $\mathcal{A}$ is in $\mathcal{U}$.
    \item The first point $X$ reaches outside of $A_-$ is in $\mathcal{A}_r$. Later, the first point that $X$ reaches outside of $\mathcal{U}$ is in $\mathcal{U}$.
\end{itemize}
Using this decomposition, we break up $f(-x_0)$ (the quantity we seek to put an upper bound on) into three additive terms:
\begin{align*}
    f(-x_0) &= \mathbb{P}_{-x_0} \left( X_{\tau_{\mathcal{A}}} \in \mathcal{U} \right)\\
    &= \mathbb{P}_{-x_0} \left( X_{\tau_{A_-}} \in \mathcal{U} \right) + \mathbb{P}_{-x_0} \left( X_{\tau_{A_-}} \in \mathcal{A}_\ell, X_{\tau_{\mathcal{A}}} \in \mathcal{U} \right) + \mathbb{P}_{-x_0} \left( X_{\tau_{A_-}} \in \mathcal{A}_r, X_{\tau_{\mathcal{A}}} \in \mathcal{U} \right).
\end{align*}
By \eqref{GEHEbound}, this becomes
\begin{equation} \label{threeTerms}
    f(-x_0) \leq \mathbb{P}_{-x_0} \left( X_{\tau_{A_-}} \in \mathcal{U} \right) + G(\mathcal{A}_\ell) H(\mathcal{A}_\ell) + G(\mathcal{A}_r) H(\mathcal{A}_r).
\end{equation}

We will handle each term from \eqref{threeTerms} separately.
First, we use \eqref{LsfApplied} (the equation we derived from the L\'{e}vy system formula) to put an upper bound on $\mathbb{P}_{-x_0} \left( X_{\tau_{A_-}} \in \mathcal{U} \right)$. Then we apply Lemma \ref{IntermediateJumpE1E2} in two different ways to derive upper bounds on $G(\mathcal{A}_\ell)$ and $H(\mathcal{A}_r)$. We use a short argument involving the translation-invariance of the process to put an upper bound on $G(\mathcal{A}_r)$. Finally, we use Lemma \ref{PreferredSideW} (one of the ``Preferred side" lemmas from Section \ref{PreferredSideSection}) to put an upper bound on $H(\mathcal{A}_\ell)$.

\subsection{\texorpdfstring{Using the L\'{e}vy system formula to obtain an upper bound on $\mathbb{P}_{-x_0}(X_{\tau_{A_-}} \in \mathcal{U})$}{TEXT}}

Let us first handle the term $\mathbb{P}_{-x_0} \left( X_{\tau_{A_-}} \in \mathcal{U} \right)$. We will compare it to $\mathbb{P}_{x_0} \left( X_{\tau_{\mathcal{A}}} \in \mathcal{U} \right)$, using the L\'{e}vy system formula.

\begin{lemma} \label{MainLsfFormula}
If $\alpha$, $c$, $\mathcal{A}$, $\mathcal{U}$, $f$, and $x_0$ are as in Proposition \ref{DiagramProp}, then
\begin{equation} \label{MainLsfUseEquation}
    \mathbb{P}_{-x_0} \left( X_{\tau_{\mathcal{A}}} \in \mathcal{U} \right) \leq \alpha^{-d} \frac{j(R)}{j(cR)} f(x_0).
\end{equation}
\end{lemma}

\begin{proof}
The first coordinate of each point in $A_-$ is at most $(-1+2\alpha) R$, and the first coordinate of any point in $U$ is at least $(1+\alpha)R$. Therefore, for all $y \in A_-$ and $z \in \mathcal{U}$, $|z-y| \geq (2-\alpha) R \geq R$, so $J(y, z) \leq j(R)$. Since this holds for all $z \in \mathcal{U}$,
\begin{equation} \label{Jnegau}
    J(y, \mathcal{U}) \leq |\mathcal{U}| \, j(R) \qquad\mbox{for all $y \in A_-$}.
\end{equation}
Similarly, for all $y \in A_+$ and $z \in \mathcal{U}$, $|z-y| < \sqrt{5^2+2^2} \alpha R = \sqrt{29}\alpha R \leq cR$, so $J(y, z) \geq j(cR)$. Since this holds for all $z \in \mathcal{U}$,
\begin{equation} \label{Jau}
    J(y, \mathcal{U}) \geq J(y, U) \geq |U| \, j(cR) = \alpha^d |\mathcal{U}| \, j(cR) \qquad\mbox{for all $y \in A_+$}.
\end{equation}

By \eqref{LsfApplied} and \eqref{Jnegau},
\begin{align}
    \mathbb{P}_{-x_0} \left( X_{\tau_{A_-}} \in \mathcal{U} \right) &= \mathbb{E}_{-x_0} \left[ \int_0^{\tau_{A_-}} J(X_s, \mathcal{U}) \, ds \right] \quad\mbox{(by \eqref{LsfApplied})} \nonumber\\
    &\leq \mathbb{E}_{-x_0} \left[ \int_0^{\tau_{A_-}} \, ds \right] \cdot |\mathcal{U}| \, j(R) \quad\mbox{(by \eqref{Jnegau})} \nonumber\\
    &=\mathbb{E}_{-x_0} \left[ \tau_{A_-} \right] \, |\mathcal{U}| \, j(R). \label{MainLsfUseEquation1}
\end{align}
Applying \eqref{LsfApplied} and \eqref{Jau} similarly gives
\begin{equation} \label{MainLsfUseEquation2}
    \mathbb{P}_{x_0} \left( X_{\tau_{\mathcal{A}}} \right) \geq \mathbb{E}_{x_0} \left[ \tau_{A_+} \right] \, \alpha^d |\mathcal{U}| \, j(cR).
\end{equation}

By \eqref{MainLsfUseEquation1} and \eqref{MainLsfUseEquation2},
\begin{equation} \label{MainLsfUseEquation3}
    \frac{\mathbb{P}_{-x_0}\left( X_{\tau_{A_-}} \in \mathcal{U} \right)}{\mathbb{P}_{x_0}\left( X_{\tau_{A_+}} \in \mathcal{U} \right)} \leq \frac{\mathbb{E}_{-x_0} \left[ \tau_{A_-} \right]}{\mathbb{E}_{x_0} \left[ \tau_{A_+} \right]} \frac{|\mathcal{U}|}{\alpha^d |\mathcal{U}|} \frac{j(R)}{j(cR)}.
\end{equation}
Since $X$ is translation-invariant, $\mathbb{E}_{-x_0} \left[ \tau_{A_-} \right] = \mathbb{E}_{x_0} \left[ \tau_{A_+} \right]$. Also, since $A_+ \subseteq \mathcal{A}$, we have $\mathbb{P}_{x_0} (X_{\tau_{A_+}} \in \mathcal{U}) \leq \mathbb{P}_{x_0}(X_{\tau_{\mathcal{A}}} \in \mathcal{U}) = f(x_0)$. Therefore, \eqref{MainLsfUseEquation3} implies \eqref{MainLsfUseEquation}.
\end{proof}

\subsection{\texorpdfstring{Upper bounds on $G(\mathcal{A}_\ell)$, $H(\mathcal{A}_r)$, and $G(\mathcal{A}_r)$}{TEXT}}

Next, let us use Lemma \ref{IntermediateJumpE1E2} to put upper bounds on both $G(\mathcal{A}_\ell)$ and $H(\mathcal{A}_r)$.

\begin{lemma} \label{IntermediateJumpHarGal}
Both $G(\mathcal{A}_\ell)$ and $H(\mathcal{A}_r)$ are less than or equal to $\mathbb{P}_0 \left( X_{\tau_{B(0, \alpha R)}} \in B(0, 10 R) \right)$.
\end{lemma}

\begin{proof}
By applying Lemma \ref{IntermediateJumpE1E2} with $x=-x_0$, $E_1=A_-$, $E_2=\mathcal{A}_\ell$, $r_1=\alpha R$, and $r_2=10 R$,
\begin{equation*}
    G(\mathcal{A}_\ell) = \mathbb{P}_{-x_0} \left( X_{\tau_{A_-}} \in \mathcal{A}_\ell \right) \leq \mathbb{P}_0 \left( X_{\tau_{B(0, \alpha R)}} \in B(0, 10 R ) \right).
\end{equation*}
For all $x \in \mathcal{A}$, by applying Lemma \ref{IntermediateJumpE1E2} with $E_1=\mathcal{A}$, $E_2=\mathcal{U}$, $r_1=\alpha R$, and $r_2=10R$, we have
\begin{equation*}
    f(x) = \mathbb{P}_x (X_{\tau_{\mathcal{A}}} \in \mathcal{U} ) \leq \mathbb{P}_0 \left( X_{\tau_{B(0, \alpha R)}} \in B(0, 10 R ) \right) \qquad\mbox{for all $x \in A$}.
\end{equation*}
This means that
\begin{equation} \label{HEBound}
    H(E) = \max_{y \in E} f(y) \leq \mathbb{P}_0 \left( X_{\tau_{B(0, \alpha R)}} \in B(0, 10 R ) \right) \qquad\mbox{for all measurable $E \subseteq \mathcal{A} \setminus A_-$}.
\end{equation}
In particular, by applying \eqref{HEBound} to $E=\mathcal{A}_r$,
\begin{equation*}
    H(\mathcal{A}_r) \leq \mathbb{P}_0 \left( X_{\tau_{B(0, \alpha R)}} \in B(0, 10 R ) \right).
\end{equation*}
\end{proof}

To put an upper bound on $G(\mathcal{A}_r)$, we simply use the translation-invariance of the process $X$.

\begin{lemma} \label{GArBound}
$G(\mathcal{A}_r) \leq f(x_0)$.
\end{lemma}

\begin{figure}[ht]
\begin{framed}
\captionof{figure}{A visual aide for the proof of Lemma \ref{GArBound}. Given two different versions of the process $X$, with initial points $-x_0$ and $x_0$ (respectively) and identical increments, if $X_{\tau_{A_-}} \in \mathcal{A}_r$ for the first process, then $X_{\tau_{\mathcal{A}}} \in (\mathcal{A}_r + 2x_0) \subseteq U$ for the second process.}
\label{ParallelJumpsFigure}
\centering
\begin{tikzpicture}[scale=0.7]
    \draw[color=blue] (-3, -3) -- (3, -3) -- (3, 3) -- (-3, 3) -- (-3, -3);
    
    \draw[color=blue] (3.5, -3) -- (9.5, -3) -- (9.5, 3) -- (3.5, 3) -- (3.5, -3);
    
    \draw[color=red] (-3, -.5) -- (-2, -.5) -- (-2, .5) -- (-3, .5) -- (-3, -.5);
    
    \draw[color=red] (3, -.5) -- (2, -.5) -- (2, .5) -- (3, .5) -- (3, -.5);
    
    
    
    \draw[color=purple] (0, 2.94) -- (2.94, 2.94) -- (2.94, -2.94) -- (0, -2.94) -- (0, 2.94);
    
    \draw[color=purple] (5, 2.94) -- (7.94, 2.94) -- (7.94, -2.94) -- (5, -2.94) -- (5, 2.94);
    
    \node[shape=circle, fill=black, scale=0.1, label=left:$0$] at (0, 0) {$0$};
    
    
    
    \node[color=red] at (-2.5, 0) {$A_-$};
    
    \node[color=red] at (2.5, 0) {$A_+$};
    
    
    \node[color=purple] at (1.5, 2.7) {$\mathcal{A}_r$};
    \node[color=purple] at (6.5, 2.7) {$\mathcal{A}_r + 2x_0$};
    
    \node[color=blue] at (0, -3.3) {$\mathcal{A}$};
    
    \node[color=blue] at (6.5, -3.3) {$\mathcal{U}$};
    
    \draw[color=cyan] (-2.6, 0.4) -> (0.1, 2);
    \draw[color=orange] (-2.5, 0) -> (1.5, 2);
    \draw[color=green] (-2.2, -0.3) -> (2.9, 2);
    
    \draw[color=cyan] (2.4, 0.4) -> (5.1, 2);
    \draw[color=orange] (2.5, 0) -> (6.5, 2);
    \draw[color=green] (2.8, -0.3) -> (7.9, 2);
\end{tikzpicture}
\end{framed}
\end{figure}
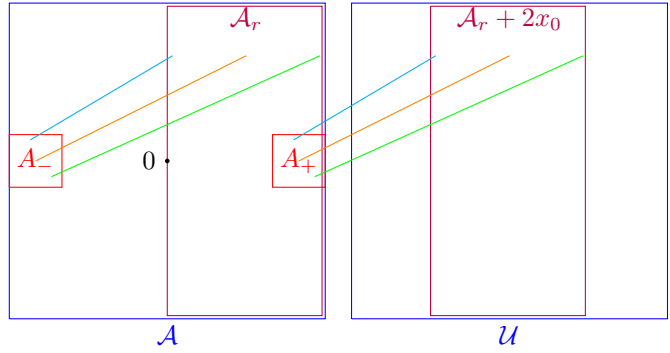

\begin{proof}
Suppose two different versions of $X$, one starting at $-x_0$ and the other starting at $x_0$, have identical increments. In order to have $X_{\tau_{A_-}} \in \mathcal{U}$ for the first process, we must have $X_{\tau_{A_+}} \in \mathcal{U}$ for the second process, as demonstrated Figure \ref{ParallelJumpsFigure}.

Therefore,
\begin{equation*}
    G(\mathcal{A}_r) = \mathbb{P}_{-x_0} \left( X_{\tau_{A_-}} \in \mathcal{A}_r \right) \leq \mathbb{P}_{x_0} \left( X_{\tau_{A_+}} \in \mathcal{U} \right) \leq f(x_0).
\end{equation*}
\end{proof}

\subsection{\texorpdfstring{Using the ``Preferred side" lemmas to obtain an upper bound on $H(\mathcal{A}_\ell)$}{TEXT}}

All that remains is to show that $H(\mathcal{A}_\ell) = \max_{y \in \mathcal{A}_\ell} \mathbb{P}_y \left( X_{\tau_{\mathcal{A}}} \in \mathcal{U} \right)$ is bounded above by $f(x)$ for some $x \in B(0, (1-\alpha/2) R)$. Intuitively, one would think that $f(x)$ should be increasing as $x$ gets closer to $\mathcal{U}$, and therefore, $H(\mathcal{A}_\ell)$ should easily be less than $f(x_0)$, as $x_0$ is much closer to $\mathcal{U}$ than any point in $\mathcal{A}_\ell$. Unfortunately, it proved surprisingly hard to come up with a proof that encapsulated this idea. We were however able to find a proof that 
$f(y) \leq f(0)$ for all $y \in \mathcal{A}_\ell$.
This proof depends on the following lemma, and Lemma \ref{PreferredSideW} (the technical ``Preferred side" lemma proved in Section \ref{PreferredSideSection}).

\begin{lemma} \label{CenteringIncreasesJU}
Suppose $y=(y_1, \tilde{y})$ and $z=(y_1, \tilde{z})$ are two points in $\mathcal{A}$ that share the same first coordinate, with $y_1 \in (-R, R)$ and $\tilde{y}, \tilde{z} \in B^{(d-1)}(0, R)$, and suppose that $|\tilde{y}| \leq |\tilde{z}|$. Then $J(y, \mathcal{U}) \geq J(z, \mathcal{U})$.
\end{lemma}

\begin{proof}
By the symmetry in the geometry of $\mathcal{A}$ and $\mathcal{U}$, we can assume without loss of generality that $y$ and $z$ are of the form
\begin{equation*}
    y=(y_1, y_2, 0, \dots, 0), \qquad z=(y_1, z_2, 0, \dots, 0)
\end{equation*}
where $z_2 \geq y_2 \geq 0$. (The values of $J(y, \mathcal{U})$ and $J(z, \mathcal{U})$ do not change if we replace $y$ with $(y_1, |\tilde{y}|, 0, \dots, 0)$, and $z$ with $(y_1, |\tilde{z}|, 0, \dots, 0)$.)

Recall that $\mathcal{U} = \left((1+\alpha)R, (3+\alpha)R \right) \times B^{(d-1)}(0, R)$.
Let
\begin{equation*}
    \Uhat := \left((1+\alpha)R, (3+\alpha)R \right) \times B^{(d-1)}\left((0, z_2-y_2, 0, \dots, 0), R \right).
\end{equation*}
Note that $\Uhat=\mathcal{U}+(z-y)$; in other words, $\Uhat$ is the result of shifting $\mathcal{U}$ by $z-y$. By translation-invariance,
\begin{equation} \label{yU=zUhat}
    J(y, \mathcal{U}) = J(z, \Uhat).
\end{equation}
Thus,
\begin{align}
    J(y, \mathcal{U})-J(z, \mathcal{U}) &= J(z, \Uhat)-J(z, \mathcal{U}) \quad\mbox{(by \eqref{yU=zUhat})} \nonumber\\
    &= \left( J(z, \mathcal{U} \cap \Uhat ) + J(z, \Uhat \setminus \mathcal{U} ) \right) - \left( J(z, \mathcal{U} \cap \Uhat ) + J(z, \mathcal{U} \setminus \Uhat ) \right) \nonumber\\
    &= J(z, \Uhat \setminus \mathcal{U}) - J(z, \mathcal{U} \setminus \Uhat ) \nonumber\\
    &= \int_{\Uhat\setminus\mathcal{U}} J(z, u) \, du - \int_{\mathcal{U}\setminus\Uhat} J(z, u). \label{JdifferenceIntermsofUdifferences}
\end{align}

Let us consider the following bijection $f : \mathcal{U} \setminus \Uhat \to \Uhat \setminus \mathcal{U}$, defined by
\begin{equation*}
    f(u_1, u_2, u_3, \dots, u_d) := (u_1, z_2-y_2-u_2, u_3, \dots, u_d).
\end{equation*}
This is clearly a bijection, as its inverse is given by the same formula. It also preserves measure. Therefore, \eqref{JdifferenceIntermsofUdifferences} can be re-written as
\begin{equation} \label{JdifferenceAsIntegralOfDifferences}
    J(y, \mathcal{U})-J(z, \mathcal{U}) = \int_{\mathcal{U}\setminus\Uhat} \left( J(z, f(u)) - J(z, u) \right) \, du.
\end{equation}

We will show that $J(z, f(u)) \geq J(z, u)$ for all $u \in \mathcal{U}\setminus\Uhat$. This will allow us to prove that the quantity in \eqref{JdifferenceAsIntegralOfDifferences} is non-negative. Fix $u=(u_1, u_2, u_3, \dots, u_d) \in \mathcal{U}\setminus\Uhat$. By the definitions of $\mathcal{U}$ and $\Uhat$,
\begin{equation*}
    \sqrt{u_1^2+u_2^2+u_3^2+\cdots+u_d^2} < R \qquad\mbox{and}\qquad \sqrt{u_1^2+(z_2-y_2-u_2)^2+u_3^2+\cdots+u_d^2} \geq R.
\end{equation*}
Thus,
\begin{equation} \label{CompareSecondCoordinates}
    |u_2| < |u_2-(z_2-y_2)|.
\end{equation}
It is a fact that for any $a>0$, the set  $\{x:|x|<|x-a|\}$ is equal to $\{x:x\leq \frac{a}{2}\}$. (To see this, graph the functions $|x|$ and $|x-a|$.) Therefore, \eqref{CompareSecondCoordinates} is equivalent to
\begin{equation}\label{u2}
    u_2 < \frac12(z_2-y_2).
\end{equation}
It follows from \eqref{u2} that
\begin{equation*}
    z_2-u_2 > \frac12(z_2+y_2)
\end{equation*}
and
\begin{equation*}
    y_2+u_2 < \frac12(z_2+y_2).
\end{equation*}
Thus,
\begin{equation}\label{z2-u21}
    z_2-u_2 > y_2+u_2.
\end{equation}
Recall that $z_2 \geq 0 \geq -y_2$. Thus,
\begin{equation} \label{z2-u22}
    z_2-u_2 \geq -y_2-u_2.
\end{equation}
By \eqref{z2-u21} and \eqref{z2-u22},
\begin{equation}\label{z2-u2}
    z_2-u_2 \geq |y_2+u_2|.
\end{equation}
Therefore,
\begin{align}
    J(z, f(u)) &= J \Big( (y_1, z_2, 0, \dots, 0), (u_1, z_2-y_2-u_2, u_3, \dots, u_d) \Big) \nonumber\\
    &= j\left( \sqrt{(y_1-u_1)^2 + (y_2+u_2)^2 + u_3^2 + \cdots + u_d^2}\right) \nonumber\\
    &\leq j\left( \sqrt{(y_1-u_1)^2 + (z_2-u_2)^2 + u_3^2 + \cdots + u_d^2}\right) \quad\mbox{(by \eqref{z2-u2})} \nonumber\\
    &= J \Big( (y_1, z_2, 0, \dots, 0), (u_1, u_2, u_3, \dots, u_d) \Big) \nonumber\\
    &= J(z, u).\label{DifferencesAllNeg}
\end{align}

By \eqref{JdifferenceAsIntegralOfDifferences} and \eqref{DifferencesAllNeg},
\begin{equation*}
    J(y, \mathcal{U}) - J(z, \mathcal{U}) \leq 0.
\end{equation*}

\end{proof}
\FloatBarrier

\begin{lemma} \label{HalBound}
$H(\mathcal{A}_\ell) \leq f(0)$.
\end{lemma}

\begin{figure}
\begin{framed}
\captionof{figure}{Proof of Lemma \ref{HalBound}}
\label{HalProofFigure}
\centering
\begin{tikzpicture}[scale=0.7]
    \draw[color=blue] (-3, -3) -- (3, -3) -- (3, 3) -- (-3, 3) -- (-3, -3);
    
    \draw[color=blue] (3.5, -3) -- (9.5, -3) -- (9.5, 3) -- (3.5, 3) -- (3.5, -3);
    
    \draw[color=blue] (0, 3) -- (0, 1.5);
    \draw[color=blue] (0, 0) -- (0, -3);
    
    \node[color=blue] at (0, -3.3) {$\mathcal{A}$};
    
    \node[color=blue] at (6.5, -3.3) {$\mathcal{U}$};
    
    \node[shape=circle, fill=black, scale=0.3, label=right:$0$] (0) at (0, 0) {$0$};
    
    \node[shape=circle, fill=black, scale=0.3, label=left:$y$] (y) at (-2, 1.5) {$y$};
    
    \node[shape=circle, fill=black, scale=0.3, label=right:$z$] (z) at (0, 1.5) {$z$};
    
    \draw[color=red] (-1, 4) -- (-1, -4);
    \node[color=red] at (-1, -4.5) {$H_1$};
    \node[color=red] at (-1.5, -2.5) {$W_1$};
    \node[color=red] at (-.5, -2.5) {$V_1$};
    
    \draw[color=red] (-4, .75) -- (10.5, .75);
    \node[color=red] at (-4.5, .75) {$H_2$};
    \node[color=red] at (2, 1.25) {$W_2$};
    \node[color=red] at (2, .25) {$V_2$};
    
    \draw[dashed] (y) -- (z);
    \draw[dashed] (0) -- (z);
\end{tikzpicture}
\end{framed}
\end{figure}

\begin{proof}
We would like to show that $\sup_{y \in \mathcal{A}_\ell} f(y) \leq f(0)$. Fix $y \in \mathcal{A}_\ell$.

By the symmetry in the geometry of $\mathcal{A}$ and $\mathcal{U}$, we can us assume without loss of generality that $y=(y_1, y_2, 0, \dots, 0)$ for some $y_2 \geq 0$. (The value of $f(y)$ does not change if $y=(y_1, y_2, y_3, \dots, y_d)$ is replaced with $(y_1, \sqrt{y_2^2+y_3^2+\cdots y_d^2}, 0, \dots, 0)$.)

Let $z=(0, y_2, 0, \dots, 0)$. Let $H_1$ be the hyperplane consisting of all points that are at an equal distance from $y$ and $z$. Let $H_2$ be the hyperplane consisting of all points that are at an equal distance from $z$ and $0$.

In order to follow this proof, we recommend the reader look at Figure \ref{HalProofFigure}.

We will first apply Lemma \ref{PreferredSideW} to $H=H_1$, $D=\mathcal{A}$, and $g(x)=J(x, \mathcal{U})$, where $V=V_1$ (the preferred side) is the side of $H_1$ containing $z$, and $W=W_1$ is the side containing $y$. We must verify \eqref{DUlargerThanDV} and \eqref{GLargerOnU}.

Let us start with \eqref{DUlargerThanDV}. Fix $w \in \mathcal{A} \cap W_1$. We must show that $w'$ (the reflection of $w$ across $H_1$) is also in $\mathcal{A}$. We can write $w$ as $w=(w_1, \tilde{w})$ for some $w_1 \in (-R, y_1/2]$ and $\tilde{w} \in B^{(d-1)}(0, R)$. This means that $w'=(y_1-w_1, \tilde{w})$. Recall that $-R < w_1 \leq y_1/2 \leq 0$, so
\begin{equation*}
    y_1-w_1 \geq y_1 > -R
\end{equation*}
and
\begin{equation*}
    y_1-w_1 \leq 0-(-R) = R.
\end{equation*}
Since $-R < y_1-w_1 < R$ and $\tilde{w} \in B^{(d-1)}(0, R)$, we have $w'=(y_1-w_1, \tilde{w}) \in (-R, R) \times B^{(d-1)}(0, R) = \mathcal{A}$, as desired.

Now let us check \eqref{GLargerOnU}. Given $w \in \mathcal{A} \cap W_1$ and $w' \in \mathcal{A} \cap V_1$, every $x \in \mathcal{U}$ is closer to $w'$ than $w$. Since $j(r)$ is non-increasing, this means that $g(w)=J(w, \mathcal{U}) \leq J(w', \mathcal{U}) = g(w')$.

Now that we have verified its conditions, Lemma \ref{PreferredSideW} tells us that
\begin{equation} \label{YzComparison}
    f(y) \leq f(z).
\end{equation}

Next, we will make a similar use of Lemma \ref{PreferredSideW} in order to show that $f(z) \leq f(0)$. Let $H=H_2$, $D=\mathcal{A}$, and $g(x)=J(x, \mathcal{U})$. Let $V=V_2$ (the preferred side) be the side of $H_2$ containing $0$, and let $W=W_2$ be the side containing $z$. We must again verify \eqref{DUlargerThanDV} and \eqref{GLargerOnU}.

Recall that $y=(y_1, y_2, 0, \dots, 0)$ and $z=(0, y_2, 0, \dots, 0)$.
Then $H_2$ (the hyperplane of points equidistant from $0$ and $z$) is equal to the set of $(x_1, x_2, \dots, x_d)$ such that $x_2 = y_2/2$. Furthermore, $V_2$ is the set of $(x_1, x_2, \dots, x_d)$ such that $x_2 \geq y_2/2$ and $W_2$ is the set of $(x_1, x_2, \dots, x_d)$ such that $x_2 < y_2/2$.

Fix $w=(w_1, w_2, w_3, \dots, w_d) \in \mathcal{A} \cap W_2$. Since $H=H_2$ is the hyperplane of points whose second coordinate is $y_2/2$, the reflection of $w$ across $H_2$ is $w'=(w_1, y_2-w_2, w_3, \dots, w_d)$. In order to verify \eqref{DUlargerThanDV} and \eqref{GLargerOnU}, we must check that $w' \in \mathcal{A}$ and $J(w, \mathcal{U}) \leq J(w', \mathcal{U})$.

We clearly have $w_1 \in (-R, R)$, since $w \in \mathcal{A}$. Since $y_2 \geq 0$ and $w \in W_2$, we also have
\begin{align*}
    (y_2-w_2)^2 &= y_2^2 - 2y_2 w_2 + w_2^2\\
    &= y_2 (y_2-2w_2) + w_2^2\\
    &\leq w_2^2 \qquad\mbox{(since $y_2 \geq 0$ and $w_2 \leq y_2/2$)}.
\end{align*}
Thus,
\begin{equation*}
    |(y_2-w_2, w_3, \dots, w_d)| \leq |(w_2, w_3, \dots, w_d)| < R. 
\end{equation*}

Since $w_1 \in (-R, R)$ and $|(y_2-w_2, w_3, \dots, w_d)| < R$, we have
\begin{equation*}
    w'=(w_1, w_2, w_3, \dots, w_d) \in (-R, R) \times B^{(d-1)}(0, R) = \mathcal{A}
\end{equation*}
so \eqref{DUlargerThanDV} is confirmed.

Since $|(y_2-w_2, w_3, \dots, w_d)| \leq |(w_2, w_3, \dots, w_d)|$, Lemma \ref{CenteringIncreasesJU} tells us that $J(w', \mathcal{U}) \geq J(w, \mathcal{U})$, so \eqref{GLargerOnU} is confirmed.

By Lemma \ref{PreferredSideW},
\begin{equation}\label{Z0Comparison}
    f(z) \leq f(0).
\end{equation}

By \eqref{YzComparison} and \eqref{Z0Comparison}, $f(y) \leq f(0)$.
\end{proof}

\FloatBarrier

We have now proven upper bounds for all of the necessary terms, and are ready to complete the proof of Proposition \ref{DiagramProp}.

\begin{proof}[Proof of Proposition \ref{DiagramProp}]
Recall that $x_0$ and $0$ belong to $B(0, (1-\alpha/2)R)$.
If we start with \eqref{threeTerms} and then replace each of the terms $\mathbb{P}_{-x_0}\left( X_{\tau_{A_-}} \in \mathcal{U} \right)$, $G(\mathcal{A}_\ell)$, $H(\mathcal{A}_\ell)$, $G(\mathcal{A}_r)$, and $H(\mathcal{A}_r)$ with its respective upper bound from Lemmas \ref{MainLsfFormula}-\ref{GArBound} and \ref{HalBound}, we obtain
\begin{align*}
    f(-x_0) \leq & \alpha^{-d} \frac{j(R)}{j(cR)} f(x_0)\\
    &+ \mathbb{P}_0 \left( X_{\tau_{B(0, \alpha R)}} \in B(0, 10 R ) \right) \cdot f(0)\\
    &+ f(x_0) \cdot \mathbb{P}_0 \left( X_{\tau_{B(0, \alpha R)}} \in B(0, 10 R ) \right)\\
    \leq & \left( \alpha^{-d} \frac{j(R)}{j(cR)} + 2 \mathbb{P}_0 \left( X_{\tau_{B(0, \alpha R)}} \in B(0, 10 R ) \right) \right) \sup_{B(0, (1-\alpha/2) R)} f.
\end{align*}
\end{proof}

\section{Specific counterexamples} \label{CounterexamplesSection}

In this section, we construct specific subordinated Brownian motions that satisfy the conditions of Theorem \ref{MainResult}, and therefore do not satisfy $\EHI$. 

First, let us distinguish between $\EHI$ at large scales and $\EHI$ at small scales. We say that $X$ satisfies the \textit{large-scale elliptic Harnack inequality} ($\LargeEHI$) if there exist $C_1, C_2 > 0$ and $\kappa \in (0, 1)$ such that for all $x_0 \in \mathbb{R}^d$ and $r \geq C_1$, if $h$ is a non-negative function that is harmonic on $B(x_0, r)$, then
\begin{equation}\label{LargeHarDef}
    h(x) \leq C_2 h(y) \qquad\mbox{for all $x, y \in B(x_0, \kappa r)$}.
\end{equation}
We say that $X$ satisfies the \textit{small-scale elliptic Harnack inequality} ($\SmallEHI$) if there exist $C_1, C_2 > 0$ and $\kappa \in (0, 1)$ such that for all $x_0 \in \mathbb{R}^d$ and $r \in (0, C_1]$, if $h$ is a non-negative function that is harmonic on $B(x_0, r)$, then
\begin{equation}\label{SmallHarDef}
    h(x) \leq C_2 h(y) \qquad\mbox{for all $x, y \in B(x_0, \kappa r)$}.
\end{equation}
Note that like $\EHI$, both $\LargeEHI$ and $\SmallEHI$ do not depend on the specific value of $\kappa$.

Clearly, both $\LargeEHI$ and $\SmallEHI$ are necessary for $\EHI$ to hold. We construct SBMs for which $\LargeEHI$ fails, and SBMs for which $\SmallEHI$ fails.

\subsection{\texorpdfstring{A general recipe for counterexamples to $\EHI$}{TEXT}}

The following theorem gives us a general recipe for producing SBMs which fail to satisfy $\EHI$. We will use this recipe to generate two concrete examples, Example \ref{LargeScaleDiracSumExample} (which fails to satisfy $\LargeEHI$) and Example \ref{SmallScaleDiracSumExample} (which fails to satisfy $\SmallEHI$).

Before we state the theorem, let us introduce the following notation: we say $a_n \ll b_n$ (or equivalently, $b_n \gg a_n$) whenever $(a_n)$ and $(b_n)$ are sequences of positive numbers such that $a_n = o(b_n)$.

\begin{theorem} \label{GeneralRecipe}
Let $X=(X_t)=(W(S_t))$ be a subordinated Brownian motion on $\mathbb{R}^d$, let $\mu$ be the L\'{e}vy measure of the subordinator $S$, and let $\gamma$ be the drift of $S$. Suppose $(a_n)$, $(b_n)$, and $(c_n)$ are sequences such that $0 < a_n \leq b_n \leq c_n$.
For all $n$, let
\begin{equation} \label{GeneralRecipeRnDef}
    R_n := \sqrt{a_n \left( d \log(\frac{c_n}{b_n}) + 2 \log(\frac{\mu([a_n, b_n])}{\mu((b_n, \infty))}) \right)}
\end{equation}
and
\begin{equation} \label{GeneralRecipeThetaDef}
    \theta_n := \frac{\gamma+\int_{(0, b_n]} s \, \mu(ds)}{\mu((b_n, \infty))}.
\end{equation}

If we have both
\begin{equation} \label{GeneralRecipeCondition}
    c_n \gg R_n^2 \gg \max\{ b_n, \theta_n\}
\end{equation}
and
\begin{equation} \label{cn}
    \frac{\int_{(b_n, \infty)} s^{-d/2} \, \mu(ds)}{\mu((b_n, \infty))} \leq c_n^{-d/2},
\end{equation}
then $X$ satisfies the conditions of Theorem \ref{MainResult}, and therefore $X$ does not satisfy $\EHI$.
\end{theorem}

\begin{proof}
By \eqref{GeneralRecipeCondition},
$R_n^2 \gg b_n$.
Let $f$ be the function $f(n) := R_n^2 / b_n \to \infty$.

For all $n$, recall from \eqref{JumpKernel} that
\begin{align*}
    j(R_n) &= \int_{(0, \infty)} (2\pi s)^{-d/2} \exp(-\frac{R_n^2}{2s}) \, \mu(ds), \\
    j(R_n/2) &= \int_{(0, \infty)} (2\pi s)^{-d/2} \exp(-\frac{R_n^2}{8s}) \, \mu(ds).
\end{align*}
We will use the fact that $R_n^2 \gg b_n$ to show that $[a_n, b_n]$ contributes more to these integrals than $(b_n, \infty)$, and that $\left[(2\pi s)^{-d/2} \exp(-\frac{R_n^2}{8s}) \right] / \left[(2\pi s)^{-d/2} \exp(-\frac{R_n^2}{2s}) \right]$ is large for all $s \in [a_n, b_n]$, so the ratio between $j(R_n/2)$ and $j(R_n)$ is very large.

Note that $\theta_n$ is the expected value of $S_t$ at the instant before the first time $S$ takes a jump of size larger than $b_n$. We will show that with high probability, $|X_t|$ is at most on the order of $\sqrt{\theta_n}$ before the first such jump, and of an order much larger than $R_n$ immediately after. Since $R_n^2 \gg \theta_n$, this means that $\mathbb{P}_0 \left( X_{\tau_{B(0, R_n/20)}} \in B(0, 10 R_n) \right)$ will be small.

Note that our choice of the quantity $R_n$ is the result of reverse-engineering so that both of these arguments work.

By \eqref{GeneralRecipeRnDef} and some elementary algebraic manipulations,
\begin{equation}\label{RearrangeRn}
    \mu([a_n, b_n]) \cdot b_n^{-d/2} \cdot \exp(-\frac{R_n^2}{2a_n}) = \mu((b_n, \infty)) c_n^{-d/2}.
\end{equation}
Note that for all $s \in [a_n, b_n]$, we have $s^{-d/2} \geq b_n^{-d/2}$ and $\exp(-\frac{R_n^2}{2s}) \geq \exp(-\frac{R_n^2}{2a_n})$. Thus,
\begin{equation} \label{[an,bn]estimate}
    \int_{[a_n, b_n]} s^{-d/2} \exp(-\frac{R_n^2}{2s}) \, \mu(ds) \geq \mu([a_n, b_n]) \cdot b_n^{-d/2} \cdot \exp(-\frac{R_n^2}{2a_n}).
\end{equation}
Combining the above estimates, we obtain
\begin{align} \label{[an,bn]ContributesMore}
\begin{split}
    \int_{[a_n, b_n]} s^{-d/2} \exp(-\frac{R_n^2}{2s}) \, \mu(ds) &\geq \mu([a_n, b_n]) \cdot b_n^{-d/2} \cdot \exp(-\frac{R_n^2}{2a_n})  \quad\mbox{(by  \eqref{[an,bn]estimate})}\\
    &= \mu((b_n, \infty)) c_n^{-d/2} \quad\mbox{(by \eqref{RearrangeRn})} \\
    &\geq \int_{(b_n, \infty)} s^{-d/2} \, \mu(ds) \quad\mbox{(by \eqref{cn})} \\
    &\geq \int_{(b_n, \infty)} s^{-d/2} \exp(-\frac{R_n^2}{2s}) \, \mu(ds).
\end{split}
\end{align}

For all $s \leq b_n$,
\begin{equation} \label{RatioOfExps}
    \frac{\exp(-\frac{R_n^2}{8s})}{\exp(-\frac{R_n^2}{2s})} = \exp(\frac{3}{8} \cdot \frac{R_n^2}{s}) \geq \exp(\frac38 \cdot \frac{R_n^2}{b_n}) = e^{\frac{3}{8} f(n)}.
\end{equation}
Thus,
\begin{align*}
    j(R_n/2) &= \int_{(0, \infty)} (2\pi s)^{-d/2} \exp(-\frac{R_n^2}{8s}) \, \mu(ds)\\
    &\geq \int_{(0, b_n]} (2\pi s)^{-d/2} \exp(-\frac{R_n^2}{8s}) \, \mu(ds)\\
    &\geq e^{\frac38 f(n)} \int_{(0, b_n]} (2\pi s)^{-d/2} \exp(-\frac{R_n^2}{2s}) \, \mu(ds) \quad\mbox{(by \eqref{RatioOfExps})}\\
    &\geq \frac12 e^{\frac38 f(n)} \int_{(0, \infty)} (2\pi s)^{-d/2} \exp(-\frac{R_n^2}{2s}) \, \mu(ds) \quad\mbox{(by \eqref{[an,bn]ContributesMore})}\\
    &= \frac12 e^{\frac38 f(n)} j(R_n).
\end{align*}
In other words,
\begin{equation} \label{GeneralRecipeJumpRatio}
    \frac{j(R_n)}{j(R_n/2)} \leq 2e^{-\frac38 f(n)} \to 0.
\end{equation}

For all $n$, let $\tilde{S}^{(n)}= \left( \tilde{S}^{(n)}_t \right)_{t \geq 0}$ be the process
\begin{equation*}
    \tilde{S}^{(n)}_t := S_t - \sum_{0 < u \leq t : S(u)-S(u-) > b_n} \left[ S(u)-S(u-) \right].
\end{equation*}
(In other words, $\tilde{S}^{(n)}$ behaves just like $S$, but without any jumps of size larger than $b_n$.)
Then $\tilde{S}^{(n)}$ is a L\'{e}vy process, with the same drift $\gamma$ as $S$, and with L\'{e}vy measure
\begin{equation*}
    \tilde{\mu}^{(n)}(E) = \mu(E \cap (0, b_n]) \qquad\mbox{for all measurable $E \subseteq (0, \infty)$}.
\end{equation*}
It follows from the probabilistic interpretation of drift and L\'{e}vy measure that for all $t>0$,
\begin{equation*}
    \mathbb{E}\left[ \tilde{S}^{(n)}_t \right] = \left( \gamma+\int_{(0, b_n]} s \, \mu(ds) \right) t.
\end{equation*}
Let $T_n := \inf\{t: S(t)-S(t-) > b_n\}$. Then $T_n$ is exponentially-distributed with rate $\mu((b_n, \infty))$ (and therefore, with mean $1 / \mu((b_n, \infty))$). Since $T_n$ and $\tilde{S}^{(n)}$ are independent,
\begin{equation} \label{ThetanIsTheExpectation}
    \mathbb{E} \left[ \tilde{S}^{(n)}_{T_n} \right] = \left( \gamma+\int_{(0, b_n]} s \, \mu(ds) \right) \cdot \mathbb{E}[T_n] = \frac{ \gamma+\int_{(0, b_n]} s \, \mu(ds)}{\mu((b_n, \infty))} = \theta_n
\end{equation}
(where $\theta_n$ is as defined in \eqref{GeneralRecipeThetaDef}).


Recall from \eqref{GeneralRecipeCondition} that $c_n \gg R_n^2 \gg \theta_n$. Choose some sequences $(Q_n)$ and $(q_n)$ such that $c_n \gg Q_n \gg R_n^2 \gg q_n \gg \theta_n$. (For example, let $Q_n=\sqrt{c_n R_n^2}$ and $q_n=\sqrt{R_n^2 \theta_n}$.) We will show that with high probability, $S(T_n-) < q_n$ and $S(T_n) > Q_n$. This means that with high probability, $|X_t|=|W(S_t)|$ is of a smaller order than $R_n$ for all $t < T_n$, and then $|X(T_n)|$ is of a larger order than $R_n$. Therefore, $\mathbb{P}_0 \left( X_{\tau_{B(0, R_n/20)}} \in B(0, 10 R_n ) \right) \to 0$.

By \eqref{ThetanIsTheExpectation} and Markov's inequality,
\begin{equation*}
    \mathbb{P} \left( S(T_n-) \geq q_n \right) = \mathbb{P} \left( \tilde{S}^{(n)}(T_n) \geq q_n \right) \leq \frac{\theta_n}{q_n} \to 0.
\end{equation*}
Let $\Delta S(T_n) := S(T_n) - S(T_n-)$. By the probabilistic interpretation of $\mu$,
\begin{align*}
    \mathbb{P}\left( |\Delta S(T_n) | \leq Q_n \right) &=\frac{\mu((b_n, Q_n])}{\mu((b_n, \infty))}\\
    &=\frac{\mu((b_n, Q_n]) \cdot Q_n^{-d/2}}{\mu((b_n, \infty)) \cdot Q_n^{-d/2}}\\
    &\leq \frac{\int_{(b_n, Q_n]} s^{-d/2} \, \mu(ds)}{\mu((b_n, \infty)) \cdot Q_n^{-d/2}}\\
    &\leq \frac{\int_{(b_n, \infty)} s^{-d/2} \, \mu(ds)}{\mu((b_n, \infty)) \cdot Q_n^{-d/2}}\\
    &\leq \frac{c_n^{-d/2}}{Q_n^{-d/2}} \qquad\mbox{(by \eqref{cn})}\\
    &= (c_n/Q_n)^{-d/2} \to 0.
\end{align*}
Thus, with high probability, $S(T_n-) < q_n$ and $S(T_n) > Q_n$.

If $X_0=0$, by the time/space scaling of the $d$-dimensional Brownian motion $W$, we also have
\begin{equation*}
    \max_{0 \leq s \leq q_n} |W_s| < \frac{R_n}{20} \qquad\mbox{and}\qquad |X(T_n) - X(T_n-)|=|W(S(T_n)) - W(S(T_n-))| \geq 20 R_n
\end{equation*}
with high probability.

Assuming all of these high probability events occur, for all $t < T_n$ we have $|X_t|=|W(S_t)| < R_n/20$, and at time $T_n$ we have $|X(T_n)| \geq \left(20-\frac{1}{20} \right) R_n > 10 R_n$. Therefore, $\tau_{B(0, R_n/20)} = T_n$, and $X_{\tau_{B(0, R_n/20)}}\notin B(0, 10 R_n)$. This completes our proof that
\begin{equation} \label{GeneralRecipeIntermediateJumps}
    \mathbb{P}_0 \left( X_{\tau_{B(0, R_n/20)}}\in B(0, 10 R_n) \right) \to 0.
\end{equation}
By \eqref{GeneralRecipeJumpRatio} and \eqref{GeneralRecipeIntermediateJumps}, $X$ satisfies the conditions of Theorem \ref{MainResult} for $c=1/2$ and $\alpha=1/20$.
\end{proof}

\subsection{\texorpdfstring{A counterexample in which $\EHI$ fails at large scales}{TEXT}}

Let us apply Theorem \ref{GeneralRecipe} to the specific case where $\mu$ (the L\'{e}vy measure of $S$) is a sum of Dirac measures.
Suppose $(H_m)_{m=0}^\infty$ and $(A_m)_{m=0}^\infty$ are sequences of positive numbers, with $\sum_{m=0}^\infty H_m < \infty$ and $A_m \nearrow \infty$. Let $$\mu = \sum_{m=0}^\infty H_m \delta_{A_m}$$ (where $\delta_x$ denotes the Dirac measure $\delta_x(E) = \indicatorWithSetBrackets{x \in E})$).

Since $\sum_{m=0}^\infty H_m < \infty$, we have $\int_{(0, \infty)} (1 \wedge x) \, \mu(dx) \leq \mu((0, \infty)) < \infty$. Therefore, there exists a subordinator $S$ with L\'{e}vy measure $\mu$.

By applying Theorem \ref{GeneralRecipe}, with $a_n=b_n=A_n$ and $c_n=A_{n+1}$, we obtain the following corollary.

\begin{corollary} \label{IncreasingAm}
Let $\mu$ be the measure
\begin{equation*}
    \mu = \sum_{m=0}^\infty H_m \delta_{A_m}
\end{equation*}
where $(H_m)_{m=0}^\infty$ and $(A_m)_{m=0}^\infty$ are sequences of positive numbers, such that $\sum_{m=0}^\infty H_m < \infty$ and $A_m \nearrow \infty$.

Let $S=(S_t)$ be a non-decreasing, right-continuous L\'{e}vy process on $[0, \infty)$ with L\'{e}vy measure $\mu$ and drift $\gamma \geq 0$. Let $X=(X_t)=(W(S_t))$ be a SBM with subordinator $S$.

For all $n$, let
\begin{equation} \label{IncreasingAmRnDef}
    R_n := \sqrt{A_n \left( d \log(\frac{A_{n+1}}{A_n}) + 2 \log(\frac{H_n}{\sum_{m>n} H_m}) \right)}
\end{equation}
and
\begin{equation} \label{IncreasingAmThetaDef}
    \theta_n := \frac{\gamma+\sum_{m \leq n} H_m A_m}{\sum_{m>n} H_m}.
\end{equation}

If $A_{n+1} \gg R_n^2 \gg \max\{ A_n, \theta_n\}$,
then $X$ satisfies the conditions of Theorem \ref{MainResult}, and therefore $X$ does not satisfy $\LargeEHI$.
\end{corollary}

For a concrete example, consider the case when $H_m=2^{-m}$, $A_m=2^{m^2}$, and $\gamma=0$.

\begin{example} \label{LargeScaleDiracSumExample}
Let $\mu$ be the measure
\begin{equation*}
    \mu = \sum_{m=0}^\infty H_m \delta_{A_m}
\end{equation*}
where $H_m=2^{-m}$ and $A_m=2^{m^2}$.

Let $S=(S_t)$ be a non-decreasing, right-continuous L\'{e}vy process on $[0, \infty)$ with L\'{e}vy measure $\mu$ and drift $0$. Let $X=(X_t)=(W(S_t))$ be a SBM with subordinator $S$.
Then $X$ does not satisfy $\LargeEHI$.
\end{example}

\begin{proof}
Let $R_n$ and $\theta_n$ be as defined in \eqref{IncreasingAmRnDef} and \eqref{IncreasingAmThetaDef}. We must verify that $A_{n+1} \gg R_n^2 \gg \max\{A_n, \theta_n\}$. A simple calculation shows that
\begin{equation*}
    R_n^2 = 2^{n^2} \cdot \frac{d}{\log_2 e} \cdot (2n+1),
\end{equation*}
so $A_{n+1} \gg R_n^2 \gg A_n$. It remains to show that $R_n^2 \gg \theta_n$.

The value of $\theta_n$ is
\begin{align*}
    \theta_n &= \frac{\sum_{m=0}^n 2^{-m} 2^{m^2}}{2^{-n}} = 2^n \sum_{m=0}^n 2^{m^2-m}\\
    &= 2^n \left( 2^{n^2-n} + \sum_{m=0}^{n-1} 2^{m(m-1)} \right)\\
    &=2^{n^2} + 2^n \sum_{m=0}^{n-1} 2^{m(m-1)} \\
    &\leq 2^{n^2} + 2^n \cdot n \cdot 2^{(n-1)(n-2)}\\
    &= 2^{n^2} \left( 1 + n \cdot 2^{-2n+2} \right) \ll R_n^2.
\end{align*}
The quantity $R_n^2$ is on the order of $2^{n^2}n$, while $\theta_n$ is on the order of $2^{n^2}$.
\end{proof}

\subsection{\texorpdfstring{A counterexample in which $\EHI$ fails at small scales}{TEXT}}

Let us also apply Theorem \ref{GeneralRecipe} to the case where $\mu=\sum_{m=0}^\infty H_m \delta_{A_m}$, where $A_m \searrow 0$. Let $a_n=b_n=A_n$ and $c_n=A_{n-1}$.

\begin{corollary}
\label{DecreasingAm}
Let $\mu$ be the measure
\begin{equation*}
    \mu = \sum_{m=0}^\infty H_m \delta_{A_m}
\end{equation*}
where $(H_m)_{m=0}^\infty$ and $(A_m)_{m=0}^\infty$ are sequences of positive numbers, such that $\sum_{m=0}^\infty H_m A_m < \infty$ and $A_m \searrow 0$.

Let $S=(S_t)$ be a non-decreasing, right-continuous L\'{e}vy process on $[0, \infty)$ with L\'{e}vy measure $\mu$ and drift $\gamma \geq 0$. Let $X=(X_t)=(W(S_t))$ be a SBM with subordinator $S$.

For all $n$, let
\begin{equation} \label{DecreasingAmRnDef}
    R_n := \sqrt{A_n \left( d \log(\frac{A_{n-1}}{A_n}) + 2 \log(\frac{H_n}{\sum_{m<n} A_m}) \right)}
\end{equation}
and
\begin{equation} \label{DecreasingAmThetaDef}
    \theta_n := \frac{\gamma+\sum_{m \geq n} H_m A_m}{\sum_{m<n} H_m}.
\end{equation}

If $A_{n-1} \gg R_n^2 \gg \max\{ A_n, \theta_n\}$,
then $X$ satisfies the conditions of Theorem \ref{MainResult}, and therefore $X$ does not satisfy $\SmallEHI$.
\end{corollary}

This gives us an example of a SBM such that $\SmallEHI$ fails.

\begin{example} \label{SmallScaleDiracSumExample}
Let $\mu$ be the measure
\begin{equation*}
    \mu = \sum_{m=0}^\infty H_m \delta_{A_m}
\end{equation*}
where $H_m=1$ and $A_m=2^{-m^2}$.

Let $S=(S_t)$ be a non-decreasing, right-continuous L\'{e}vy process on $[0, \infty)$ with L\'{e}vy measure $\mu$ and drift $0$. Let $X=(X_t)=(W(S_t))$ be a SBM with subordinator $S$.
Then $X$ does not satisfy $\SmallEHI$.
\end{example}

\begin{proof}
The quantities from \eqref{DecreasingAmRnDef} and \eqref{DecreasingAmThetaDef} are
\begin{align*}
    R_n^2 &= 2^{-n^2} \left( d \log(\frac{2^{-n^2+2n-1}}{2^{-n^2}}) + 2\log(\frac{1}{n}) \right)\\
    &= 2^{-n^2} \left( \frac{d}{\log_2(e)} \cdot (2n-1) - 2\log(n) \right)
\end{align*}
and
\begin{equation} \label{DecreasingAmThetaSum}
    \theta_n = \frac{ \sum_{m=n}^\infty 2^{-m^2} }{n}.
\end{equation}
Since each term of the sum in \eqref{DecreasingAmThetaSum} is less than half of the last term,
\begin{equation*}
    \theta_n \leq 2 \cdot \frac{2^{-n^2}}{n}.
\end{equation*}
We therefore have $A_{n-1}\gg R_n^2 \gg \max\{A_n, \theta_n\}$. By Corollary \ref{DecreasingAm}, $X$ does not satisfy $\SmallEHI$.
\end{proof}

\subsection{A counterexample in which the subordinator has a continuous, decreasing L\'{e}vy measure}

In Examples \ref{LargeScaleDiracSumExample} and \ref{SmallScaleDiracSumExample}, the L\'{e}vy measure $\mu$ of the subordinator $S$ has large gaps: there are intervals $(a, b)$, with $0 < a \ll b$, such that $\mu((a, b))=0$, even though $\mu((0, a])>0$ and $\mu([b, \infty))>0$.

A natural question to ask is: does there exist a subordinated Brownian motion $(X_t)=(W(S_t))$, which does not satisfy $\EHI$, such that the L\'{e}vy measure of $S$ is absolutely continuous with respect to the Lebesgue measure on $(0, \infty)$, and its density is decreasing?

In this subsection, we prove that the answer to this question is ``yes."

To motivate our argument, consider the following observation. In Examples \ref{LargeScaleDiracSumExample} and \ref{SmallScaleDiracSumExample}, the subordinator $S$ only takes jumps of size $A_m$ (for some $m$). As a result, almost every jump that the process $X$ takes is on the order of $\sqrt{A_m}$ for some $m$, and there are large gaps between each $A_m$ and $A_{m+1}$. Despite this, the jump kernel $j(r)$ of $X$ is still decreasing in $r$, since it is a sum of Gaussians.

In the next example, instead of $S$ taking only jumps of size $A_m$ for some $m$, we will have $S$ take jumps of size $A_m Y$, where $Y$ is a random variable sampled from the standard exponential distribution ($\mathbb{P}(Y>y)=e^{-y}$ for all $y \geq 0$). This way, the L\'{e}vy measure of $S$ is a sum of continuous, decreasing exponentials. Each time such a jump in $S$ occurs, the process $X=(X_t)=(W(S_t))$ takes a jump with displacement $\sqrt{A_m Y}Z$, where $Z$ is a standard $d$-dimensional Gaussian independent from $Y$. Therefore, almost every jump in $X$ has magnitude on the order of $\sqrt{A_m}$, just like in Examples \ref{LargeScaleDiracSumExample} and \ref{SmallScaleDiracSumExample}, and we can use an argument very similar to the proof of Theorem \ref{GeneralRecipe} to show that $X$ does not satisfy $\EHI$.

For the rest of this subsection, let $d=1$, and let $X=(X_t)$ be SBM on $\mathbb{R}$ described as follows.

For all $m \in \mathbb{Z}_+$, let $H_m=2^{-m}$ and $A_m=2^{m^2}$. For each $m \in \mathbb{Z}_+$, consider a Poisson process $N^m=(N^m_t)_{t \geq 0}$, with rate $H_m$. Suppose all of these Poisson processes are independent.

Let $Y$ be a standard exponential random variable (so $Y$ has density $f_Y(y)=e^{-y}$ on $(0, \infty)$, and $\mathbb{P}(Y>y)=e^{-y}$ for all $y \geq 0$). Let $(Y_{m, i})_{m \in \mathbb{Z}_+, i \in \mathbb{N}}$ be a collection of independent, identically-distributed copies of $Y$, indexed by $(m, i) \in \mathbb{Z}_+ \times \mathbb{N}$.

Let $S=(S_t)_{t \geq 0}$ be the non-negative L\'{e}vy process defined by
\begin{equation*}
    S_t := \sum_{m=0}^\infty \sum_{i=1}^{N^m_t} A_m Y_{m, i}.
\end{equation*}
In other words, events in $N^m$ occur with rate $H_m$, and each time such an event occurs, the process $S$ takes a jump of size $A_m$ times a standard exponential.

Then $S$ is a non-decreasing, right-continuous L\'{e}vy process on $[0, \infty)$. For each $m$, if $Y$ is a standard exponential, then $A_m Y$ has density $A_m^{-1} e^{-x/A_m}$. Therefore, the L\'{e}vy measure $\mu$ of $S$ is given by
\begin{equation} \label{LevyMeasureContinuousExample}
    \mu(dx) = \sum_{m=0}^\infty \frac{H_m}{A_m} e^{-x/A_m} \, dx , \quad x \in (0, \infty).
\end{equation}

Let $W=(W_t)_{t \geq 0}$ be a standard Brownian motion on $\mathbb{R}$, independent of $S$, and let $X=(X_t):=(W(S_t))$ be a SBM with subordinator $S$.

We would now like to calculate the jump kernel of $X$.
The following lemma tells us the value of an integral that will come up in this calculation.

\begin{lemma} \label{IntegralThatComesUp}
For all $r \geq 0$,
\begin{equation*}
    \int_0^\infty \exp (-\left( t^2 + \frac{r^2}{t^2} \right)) \, dt = \frac{\sqrt{\pi}}{2} e^{-2r}.
\end{equation*}
\end{lemma}

\begin{proof}
Let $F(r)$ be the value of the integral in question. It is well-known that $F(0) = \frac{\sqrt{\pi}}{2}$, by the Gaussian integral. We will show that $F$ solves the differential equation $F'(r) = -2F(r)$.

By differentiating under the integral sign,
\begin{equation} \label{PreCov}
    F'(r) = \int_0^\infty \exp (-\left( t^2 + \frac{r^2}{t^2} \right)) \cdot -\frac{2r}{t^2} \, dt.
\end{equation}
(This differentiation under the integral sign is justified by Tonelli's theorem, since the difference quotient $$\frac{1}{h} \left[\exp(-\left( t^2 + (r+h)^2/t^2 \right))-\exp(-\left( t^2 + r^2/t^2 \right)) \right]$$ is non-positive for all $h > 0$.)

Let us apply the change of variables $s=r/t$ (so $ds=-r/t^2 \, dt$). Then \eqref{PreCov} becomes
\begin{align*}
    F'(r) &= - \int_0^\infty \exp(-\left( \frac{r^2}{s^2} + s^2 \right)) \cdot 2 \, ds\\
    &= -2 F(r).
\end{align*}
Thus, $F$ solves the differential equation $\{F'(r)=-2F(r); F(0)=\sqrt{\pi}/2\}$, the solution to which is $F(r) = \frac{\sqrt{\pi}}{2} e^{-2r}$.
\end{proof}

\begin{lemma}
The jump kernel of $X$ is
\begin{equation}  \label{JumpKernelXContinuousExample}
    j(r) = \sum_{m=0}^\infty H_m \sqrt{\frac{1}{2A_m}} \exp(-\sqrt{\frac{2}{A_m}}r).
\end{equation}
\end{lemma}

\begin{proof}
Consider what happens at a time $t$ when an event occurs in the Poisson process $N^m$. The subordinator $S$ takes a jump of size $A_m$ times a standard exponential: $S_t-S_{t-} \overset{d}{=} A_m Y$. This means that $X$ takes a jump with displacement $\sqrt{A_m Y}Z$, where $Y$ is a standard exponential, and $Z$ is a standard Gaussian independent from $Y$:
\begin{align*}
    X_t - X_{t-} &= W(S_t) - W(S_{t-})\\
    &= W(S_{t-} + A_m Y) - W(S_{t-})\\
    &= \sqrt{A_m Y}Z.
\end{align*}
Let us calculate the density of $\sqrt{A_m Y}Z$ at some $x \in \mathbb{R}$. Consider a small $\varepsilon>0$. By conditioning on $Y$, the probability that $\sqrt{A_m Y}Z$ is in the ball $B(x, \varepsilon)$ is equal to
\begin{align*}
    \mathbb{P}\left( \sqrt{A_m Y}Z \in B(x, \varepsilon) \right) &= \int_0^\infty e^{-y} \, \mathbb{P} \left( \sqrt{A_m y}Z \in B(x, \varepsilon) \right) \, dy\\
    &= \int_0^\infty e^{-y} \, \mathbb{P} \left(Z \in B\left(\frac{x}{\sqrt{A_m y}}, \frac{\varepsilon}{\sqrt{A_m y}}\right) \right) \, dy.
\end{align*}
Thus, the density of $\sqrt{A_m Y}Z$ at $x$ is equal to
\begin{align*}
    \lim_{\varepsilon\to 0} \frac{1}{2\varepsilon} \mathbb{P}\left( \sqrt{A_m Y}Z \in [x-\varepsilon, x+\varepsilon] \right) &= \lim_{\varepsilon\to 0} \frac{1}{2\varepsilon} \int_0^\infty e^{-y} \, \mathbb{P} \left(Z \in \left[\frac{x-\varepsilon}{\sqrt{A_m y}}, \frac{x+\varepsilon}{\sqrt{A_m y}}  \right] \right) \, dy\\
    &= \lim_{\varepsilon\to 0} \frac{1}{2\varepsilon} \int_0^\infty e^{-y} \int_{(x-\varepsilon) / \sqrt{A_m y}}^{(x+\varepsilon) / \sqrt{A_m y}} \frac{1}{\sqrt{2\pi}} e^{-z^2/2} \, dz \, dy\\
    &= \frac{1}{\sqrt{2\pi}} \int_0^\infty e^{-y} \lim_{\varepsilon\to 0} \frac{2\varepsilon/\sqrt{A_m y}}{2\varepsilon} \exp(-\frac{x^2}{2A_m y}) \, dy\\
    &= \frac{1}{\sqrt{2\pi A_m}} \int_0^\infty e^{-y} y^{-1/2} \exp(-\frac{x^2}{2A_m y}) \, dy.
\end{align*}
By the substitution $y=t^2$ (and $dy=2t \, dt$), this becomes
\begin{align*}
    \sqrt{\frac{2}{\pi A_m}} \int_0^\infty \exp(-\left(t^2 + \frac{x^2}{2 A_m t^2}\right)) \, dt.
\end{align*}
By applying Lemma \ref{IntegralThatComesUp} to $r=|x|/\sqrt{2 A_m}$, the density of $\sqrt{A_m Y}Z$ at $x$ is equal to
\begin{equation*}
    \sqrt{\frac{2}{\pi A_m}} \cdot \frac{\sqrt{\pi}}{2} \exp(-\sqrt{\frac{2}{A_m}}|x|) = \sqrt{\frac{1}{2A_m}} \exp(-\sqrt{\frac{2}{A_m}}|x|).
\end{equation*}

Recall that this was the density of a jump in $X$ corresponding to an event in $N^m$. Since these jumps occur with rate $H_m$ for all $m$, the jump kernel of $X$ is
\begin{equation*}
    j(r) = \sum_{m=0}^\infty H_m \cdot \sqrt{\frac{1}{2A_m}} \exp(-\sqrt{\frac{2}{A_m}}r).
\end{equation*}
\end{proof}

Now we can complete the proof that this SBM does not satisfy $\EHI$. We will use almost the exact same proof strategy as we did in Theorem \ref{GeneralRecipe}, but with \eqref{JumpKernelXContinuousExample} now playing the role of \eqref{JumpKernel}.

\begin{example} \label{ContinuousExample}
For all $m$, let $H_m=2^{-m}$ and $A_m=2^{m^2}$. Let $S=(S_t)$ be the L\'{e}vy process on $[0, \infty)$ such that for all $m \in \mathbb{Z}_+$, $S$ takes jumps of size $A_m$ times a standard exponential random variable, with rate $H_m$. Let $X=(X_t)$ be the SBM on $\mathbb{R}$ with subordinator $S$.

Then the L\'{e}vy measure of $S$ is continuous, and its density is decreasing. Furthermore, $X$ does not satisfy $\EHI$.
\end{example}

\begin{proof}
The density of the L\'{e}vy measure of $S$ is given by \eqref{LevyMeasureContinuousExample}. It is a sum of decreasing functions, and therefore is itself decreasing.

All that remains is to show that $X$ does not satisfy $\EHI$. To do this, we will very closely follow the proof of Theorem \ref{GeneralRecipe}. For all $n$, let
\begin{equation*}
    R_n := \sqrt{\frac{A_n}{8}} \log(\frac{A_{n+1}}{A_n}).
\end{equation*}
It follows from the definitions of $H_n$ and $R_n$ (and a bit of symbolic manipulation) that
\begin{equation*}
    H_n \sqrt{\frac{1}{2A_n}} \exp(-\sqrt{\frac{2}{A_n}} R_n) = 2^{-n} \sqrt{\frac{1}{2A_{n+1}}}
\end{equation*}
and
\begin{equation*}
    \sum_{m =n+1}^\infty H_m \sqrt{\frac{1}{2A_m}} \exp(-\sqrt{\frac{2}{A_m}} R_n) \leq \sum_{m=n+1}^\infty 2^{-m} \sqrt{\frac{1}{2A_{n+1}}} = 2^{-n} \sqrt{\frac{1}{2A_{n+1}}}.
\end{equation*}
Thus,
\begin{equation*} 
    \sum_{m =n+1}^\infty H_m \sqrt{\frac{1}{2A_m}} \exp(-\sqrt{\frac{2}{A_m}} R_n) \leq H_n \sqrt{\frac{1}{2A_n}} \exp(-\sqrt{\frac{2}{A_n}} R_n).
\end{equation*}

Therefore, when we consider the sum $j(R_n) = \sum_{m=0}^\infty H_m \sqrt{1/(2A_m)} \exp(-\sqrt{2/A_m}R_n)$, the terms from $m=0$ to $m=n$ contribute at least half of the total value:
\begin{equation}\label{NContributesMoreContinuousExample}
    \sum_{m=0}^n H_m \sqrt{\frac{1}{2A_m}} \exp(-\sqrt{\frac{2}{A_m}} R_n) \geq \frac12 \sum_{m=0}^\infty H_m \sqrt{\frac{1}{2A_m}} \exp(-\sqrt{\frac{2}{A_m}} R_n).
\end{equation}

For all $n$, let $f(n) := R_n/\sqrt{A_n} = \frac18 \log(A_{n+1} / A_n)$. Note that $A_{n+1} / A_n \to \infty$, so $f(n) \to \infty$.
For all $m \leq n$,
\begin{equation}\label{ContinuousExampleComparingExps}
    \frac{\exp(-\sqrt{\frac{2}{A_m}}\cdot\frac{R_n}{2})}{\exp(-\sqrt{\frac{2}{A_m}}R_n)} = \exp(\sqrt{\frac{2}{A_m}}\cdot\frac{R_n}{2}) \geq \exp(\sqrt{\frac{2}{A_n}}\cdot\frac{R_n}{2}) = \exp(\frac{f(n)}{\sqrt{2}})\to \infty
\end{equation}

Comparing $j(R_n/2)$ and $j(R_n)$,
\begin{align*}
    j(R_n/2) &= \sum_{m=0}^\infty H_m \sqrt{\frac{1}{2A_m}} \exp(-\sqrt{\frac{2}{A_m}} \cdot \frac{R_n}{2}) &\mbox{(by \eqref{JumpKernelXContinuousExample})}\\
    &\geq \sum_{m=0}^n H_m \sqrt{\frac{1}{2A_m}} \exp(-\sqrt{\frac{2}{A_m}} \cdot \frac{R_n}{2})\\
    &\geq \exp(\frac{f(n)}{\sqrt{2}}) \sum_{m=0}^n H_m \sqrt{\frac{1}{2A_m}} \exp(-\sqrt{\frac{2}{A_m}} R_n) &\mbox{(by \eqref{ContinuousExampleComparingExps})}\\
    &\geq \frac12 \exp(\frac{f(n)}{\sqrt{2}}) \sum_{m=0}^\infty H_m \sqrt{\frac{1}{2A_m}} \exp(-\sqrt{\frac{2}{A_m}} R_n) &\mbox{(by \eqref{NContributesMoreContinuousExample})}\\
    &= \frac12 \exp(\frac{f(n)}{\sqrt{2}}) j(R_n) &\mbox{(by \eqref{JumpKernelXContinuousExample})}.
\end{align*}
In other words,
\begin{equation} \label{ContinuousExampleJRatio}
    \frac{j(R_n)}{j(R_n/2)}\leq 2 \exp(-\frac{f(n)}{\sqrt{2}})\to 0.
\end{equation}

For all $n$, let $T_n$ be the first time that an event occurs in the Poisson process $N^m$ for any $m > n$. (In other words, $T_n$ is the first time that $S$ takes a jump of size $A_m$ times a standard exponential, for $m>n$). Let $\theta_n := \mathbb{E} \left[ S(T_n -) \right]$. It follows from the probabilistic interpretation of L\'{e}vy measure that
\begin{align*}
    \theta_n &= \frac{\sum_{m=0}^n H_m A_m}{\sum_{m=n+1}^\infty H_m} = 2^n \sum_{m=0}^n 2^{m^2-m}\\
    &=2^{n^2} \sum_{m=0}^n 2^{(m^2-m) - (n^2-n)}\\
    &= 2^{n^2} \left( 1 + \sum_{m=0}^{n-1} 2^{(m(m-1))-(n(n-1))} \right)\\
    &\leq 2^{n^2} \left( 1 + (n-1) 2^{((n-1)(n-2))-(n(n-1))} \right)\\
    &= 2^{n^2} \left(1 + (n-1)2^{-2(n-1)} \right).
\end{align*}
Thus, $\theta_n$ is on the order of $2^{n^2}$ for large $n$.

Note also that
\begin{align*}
    R_n^2 &= \frac{A_n}{8} \left( \log(\frac{A_{n+1}}{A_n}) \right)^2 = \frac{2^{n^2}}{8} \left( \log(\frac{2^{n^2+2n+1}}{2^{n^2}}) \right)^2\\
    &= \frac{2^{n^2}}{8} \left( \frac{2n+1}{\log_2 e} \right)^2
\end{align*}
so $R_n^2$ is on the order of $2^{n^2} n^2$. Recall that $A_{n+1} = 2^{(n+1)^2}$. We therefore have $A_{n+1} \gg R_n^2 \gg \theta_n$.

Choose some sequences $(Q_n)$ and $(q_n)$ such that $A_{n+1} \gg Q_n \gg R_n^2 \gg q_n \gg \theta_n$.
By the same argument as in the end of the proof of Theorem \ref{GeneralRecipe}, with high probability, $|X_t|$ is less than $\sqrt{q_n}$ for all $t < T_n$, and $|X(T_n)|$ is greater than $\sqrt{Q_n}$. We therefore have
\begin{equation} \label{ContinuousEampleIntermediateJumps}
    \mathbb{P}_0 \left( X_{\tau_{B(0, R_n/20)}} \in B(0, 10 R_n) \right) \to 0.
\end{equation}

By \eqref{ContinuousExampleJRatio} and \eqref{ContinuousEampleIntermediateJumps}, $X$ satisfies the conditions of Theorem \ref{MainResult} for $c=1/2$ and $\alpha=1/20$, so $X$ does not satisfy $\EHI$.
\end{proof}

\begin{figure}[ht]
\begin{framed}
\captionof{figure}{The ratio $j(r) / j(r/2)$ for Examples
\ref{LargeScaleDiracSumExample} and \ref{ContinuousExample}. In both cases, the $x$-axis is distorted so that $\sqrt{A_m}$ and $\sqrt{A_{m+1}}$ are the same distance away for all $m$. The points corresponding to $r=R_m$ are shown using red dots. (In all cases, $\sqrt{A_m} \leq R_m \leq \sqrt{A_{m+1}}$.) Note that $R_m$ is not necessarily the exact local minimizer, but the limit along the sequence $\{R_m\}$ is $0$.}
\label{JRatioGraphs}
\includegraphics[width=480pt]{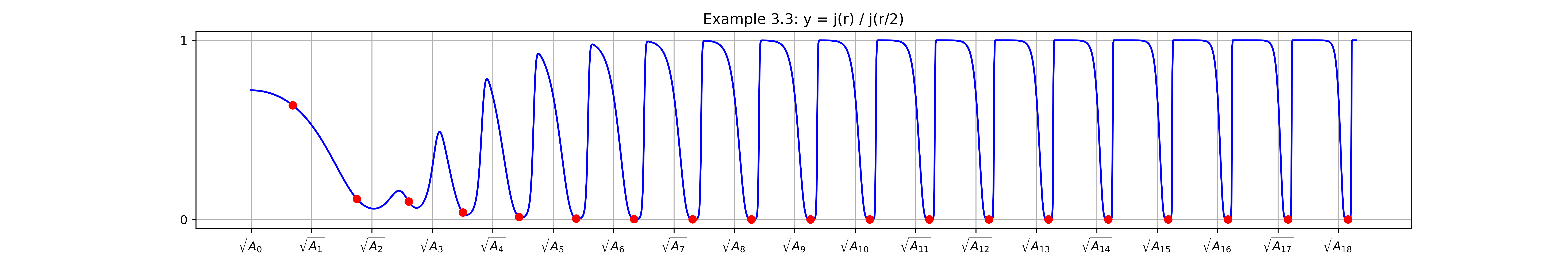}
\includegraphics[width=480pt]{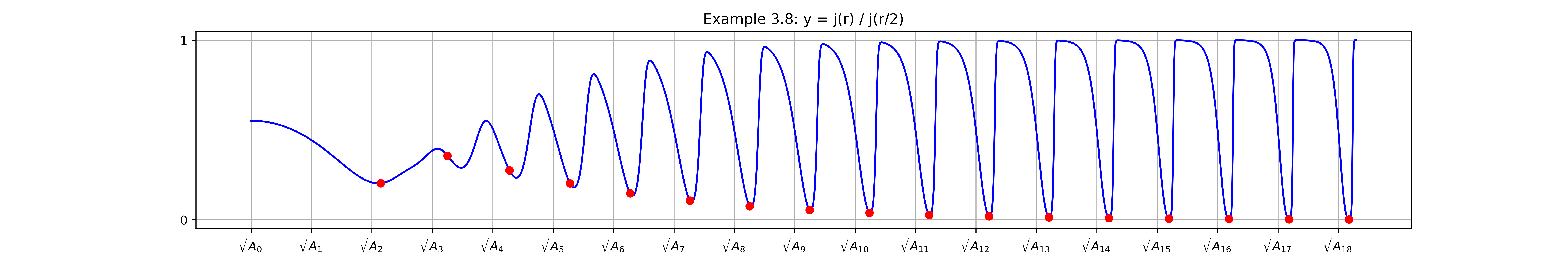}
\end{framed}
\end{figure}

\FloatBarrier

\end{document}